\documentclass{llncs}
\usepackage{amssymb,amsmath}
\usepackage{stmaryrd,paralist}
\usepackage{bbm}
\usepackage{graphics,graphicx}

\usepackage[latin1]{inputenc}
\usepackage{subfigure}
\usepackage{hyperref,color}
\usepackage[dvipsnames]{xcolor}
\graphicspath{{Figures/}} 
\graphicspath{{./Figures/}}

\usepackage{url}

\newcommand{\llncs}[1]{} %These are part of text which are removed in the LLNCS version
\newcommand{\addllncs}[1]{#1}
\newcommand{\addllncswa}[1]{#1}
\newcommand{\addllncswoa}[1]{}

\newtheorem{theo}{Theorem}
\numberwithin{theo}{section}
\newtheorem{thm}[theo]{Theorem}
\newtheorem{lem}[theo]{Lemma}

\newtheorem{prop}[theo]{Proposition}

\newtheorem{Def}[theo]{Definition}
{\medbreak}

\newcommand{\ds}{\displaystyle}

\def\Nir{N_i^{\mathrm{right}}}
\def\Nil{N_i^{\mathrm{left}}}

\newcommand{\cacher}[1]{}

\def\cA{\mathcal{A}}
\def\cB{\mathcal{B}}

\def\cO{\mathcal{O}}

\def\cW{\mathcal{W}}
\def\cL{\mathcal{L}}

\def\bL{\mathbf{L}}

\def\bO{\mathbf{O}}
\def\bW{\mathbf{W}}

\newcommand{\de}{\delta}

\newcommand{\ov}[1]{\overline{#1}}

\newcommand{\fig}[3]{\begin{figure}[h!]\begin{center}\includegraphics[#1]{#2.pdf}\end{center}\caption{#3}\label{fig:#2}\end{figure}}

\def\set5{[1:5]}
\def\omt{\omega_3}
\def\omf{\omega_5}
\def\bdelta{\boldsymbol{\delta}}
\def\ove{\overrightarrow}

\newcommand{\cwjump}{\textrm{cw-jump}}
\newcommand{\ccwjump}{\textrm{ccw-jump}}

\newcommand{\Hd}{H^\diamond}
\newcommand{\Gd}{G^\diamond}

\renewcommand{\orcidID}[1]{\href{https://orcid.org/#1}{\includegraphics[scale=.03]{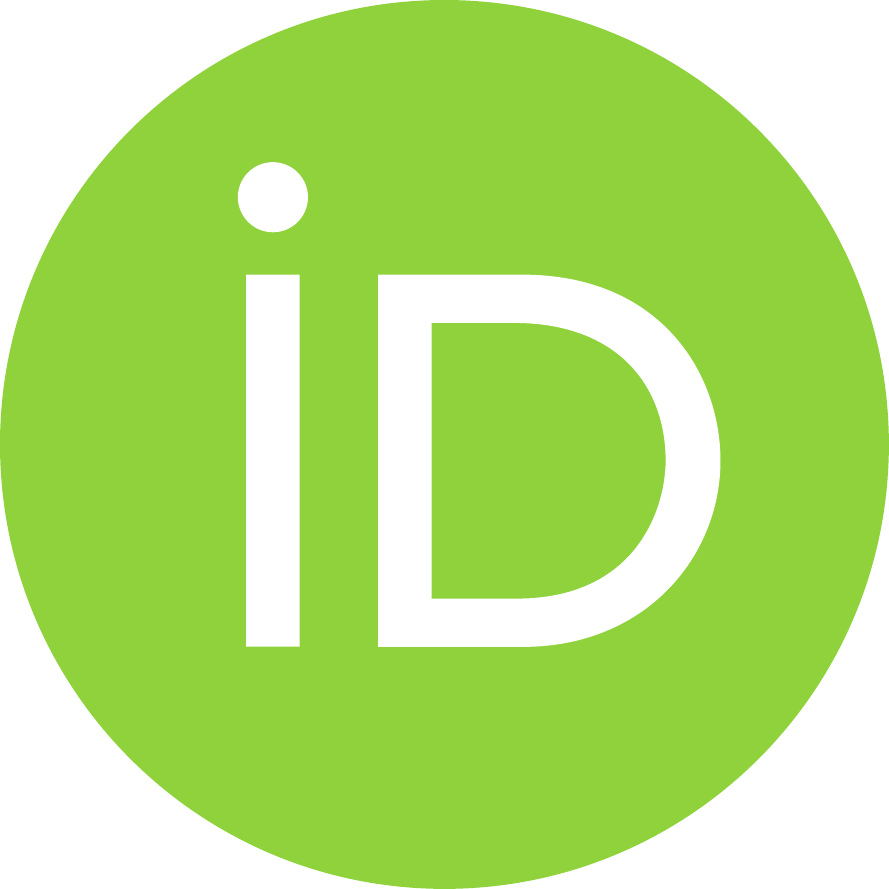}}}

\begin{document}

\author{Olivier Bernardi\inst{1}\orcidID{0000-0003-3231-7152} \and
	\'Eric Fusy\inst{2}\orcidID{0009-0000-6517-2231} \and
	Shizhe Liang\inst{1}\orcidID{0009-0004-2611-1451}
}
%
%\authorrunning{O. Aichholzer et al.}

%
\institute{Department of Mathematics, Brandeis University, Waltham MA, USA\\
	\email{\{bernardi,shizhe1011\}@brandeis.edu} \and
	CNRS/LIGM, Universit\'e Gustave Eiffel, Champs-sur-Marne, France\\
	\email{eric.fusy@univ-eiffel.fr}}

%\author[Olivier Bernardi, \'Eric Fusy, and Shizhe Liang]{Olivier Bernardi$^{*}$ \and \'{E}ric Fusy$^{\dagger}$ \and Shizhe Liang$^{+}$}
%\thanks{$^{*}$Department of Mathematics, Brandeis University, Waltham MA, USA,
%bernardi@brandeis.edu.\\% Supported by NSF grant DMS-1308441 and DMS-1800681.\\
%$^{\dagger}$LIGM/CNRS, Universit\'e Gustave Eiffel, Champs-sur-Marne, France, eric.fusy@u-pem.fr.\\
%$^{+}$Department of Mathematics, Brandeis University, Waltham MA, USA, Shizhe1011@brandeis.edu.\\
%}

% \title[Schnyder-type drawing algorithm]{A Schnyder-type drawing algorithm for 5-connected triangulations}
\title{A Schnyder-type drawing algorithm for 5-connected triangulations\thanks{OB was partially supported by NSF Grant DMS-2154242. EF was partially supported by the project ANR19-CE48-011-01 (COMBIN\'E), and the project ANR-20-CE48-0018 (3DMaps).}}
\date{\today}

\maketitle

\begin{abstract}
We define some Schnyder-type combinatorial structures on a class of planar triangulations of the pentagon which are closely related to 5-connected triangulations.
%\EF{maybe: ``which form a close superclass of 5-connected triangulations where a vertex of degree 5 is deleted"}. 
The combinatorial structures have three incarnations defined in terms of orientations, corner-labelings, and woods respectively.  
The wood incarnation consists in 5 spanning trees crossing each other in an orderly fashion. 
Similarly as for Schnyder woods on triangulations, it induces, for each vertex, a partition of the inner triangles into face-connected regions (5~regions here).
We show that the induced barycentric vertex-placement, where each vertex is at the barycenter of the 5 outer vertices with weights given by the number of faces in each region,  yields a planar straight-line drawing.
\keywords{Schnyder structures, barycentric drawing, 5-connected triangulations}
\end{abstract}

\section{Introduction}\label{sec:intro}
%\documentclass[letter]{amsart}
%\input{defs.tex}
%\author{Olivier Bernardi$^{*}$ \and \'{E}ric Fusy$^{\dagger}$ \and Shizhe Liang$^{+}$}
%\title[Grand Schnyder Woods]{Grand Schnyder Woods}
%\begin{document}

In 1989, Walter Schnyder showed that planar triangulations can be endowed with remarkable combinatorial structures, which now go by the name of \emph{Schnyder woods}~\cite{Schnyder:wood1}. A Schnyder wood of a planar triangulation 
%OB: I deleted "(drawn without crossing in the plane)" as "triangulation" already imply that. 
is a partition of its inner edges into three trees, crossing each other in an orderly manner; see Figure~\ref{fig:triangulation2}(a) for an example. In~\cite{Schnyder:wood2}, Schnyder used his structures to define an elegant algorithm to draw planar triangulations with straight edges~\cite{Schnyder:wood2}.
%\SL{(straight line segments?)}\OB{I think straight edge is fine. Otherwise we will have to say ``with edges represented by straight line segments''.}  
\llncs{The idea is to associate to each vertex $v$ three regions that are delimited by the three paths from $v$ to the root in each of the trees, and to place $v$ at the barycenter of the three fixed outer vertices, with weights given by the proportion of inner faces in each region.}

\fig{width=\linewidth}{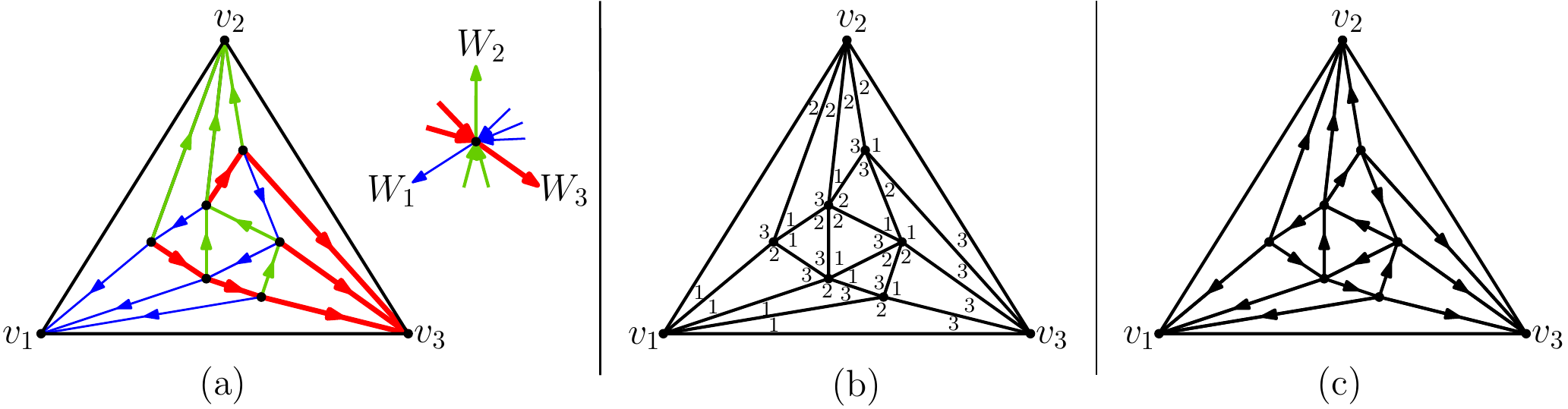}{
\addllncs{(a) A Schnyder wood (with the local condition at inner vertices), (b) the corresponding corner labeling, (c) the corresponding 3-orientation.}
\llncs{(a) A Schnyder wood of a triangulation $G$, where the three trees are indicated by three colors.
A Schnyder wood of $G$ is a partition of the inner edges of $G$ into three trees $W_1$, $W_2$, $W_3$ satisfying two conditions: 
(1) %each tree spans all the inner vertices, as well as one outer vertex which is choosen as its root, with the root of $W_1$, $W_2$, $W_3$ appearing in clockwise order around the outer face of $G$
 for all $i\in\{1,2,3\}$, the tree $W_i$ spans all the inner vertices and the outer vertex $v_i$ % which is chosen as its root, 
(2) if each tree $W_i$ is oriented toward its root $v_i$, then in clockwise direction around each inner vertex one has: the outgoing edge of $W_1$, the ingoing edges of $W_3$, the outgoing edge of $W_2$, the ingoing edges of $W_1$, the outgoing edge of $W_3$, and finally the ingoing edges of $W_2$. 
(b) Encoding the Schnyder wood by a corner labeling. (c) Encoding the Schnyder wood by a 3-orientation (orientation of the inner edges such that every inner vertex has outdegree 3).}}

In this article we study an analogue of Schnyder woods for triangulations of the pentagon, which we call \emph{5c-woods}, and we show that 
such structures exist if and only if cycles of length less than 5 have no vertex in their interior (a property closely related to 5-connectedness).  
We then use these structures to define a graph-drawing algorithm in the spirit of Schnyder's \llncs{barycentric }algorithm. 
%In the \emph{wood} incarnation of these structures, 
\llncs{A 5c-wood 
consists in five trees crossing each other in an orderly manner. The five trees yield, for each vertex $v$, five paths that delimit five regions. In our algorithm, the outer vertices are placed so as to form a regular pentagon, and each inner vertex $v$ is placed at the barycenter of the outer vertices, with weights given by the proportions of faces in each region. Our algorithm can be applied to draw 5-connected triangulations, upon deleting one of its vertices of degree 5, considered as a vertex at infinity.} 
%Compared to Tutte's spring embedding algorithm~\cite{tutte1963draw} (where planarity is guaranteed whenever the outer contour is convex), it is crucial in our proof of planarity that the outer pentagon is regular.

A disadvantage of our algorithm compared to Schnyder's original algorithm (and to algorithms such as~\cite{bonichon2007convex,chrobak1997convex,FraysseixPP88,F01,Fu07b,he1997grid,miura2022grid} using an underlying combinatorial structure or shelling order to define a vertex-placement) is that it does not yield a \emph{grid-drawing} algorithm, that is, a vertex placement where the coordinates are integers bounded by a linear function of the number of vertices of the graph.
%OB: deleted "(under any linear transform)." 
Nevertheless, we stress some nice features of our algorithm: it can be implemented in linear time, 
%OB: deleted "similarly to the face-counting algorithms in~\cite{Schnyder:wood2,F01,Fu07b}" 
it respects rotational symmetries, 
%OB: deleted "(but not systematically mirror symmetries, which the spring embedding does)," 
and\llncs{, as we will show,} %the vertices are kept at greater distance from each other than in Schnyder's drawing.
the worst-case vertex resolution (minimal distance between vertices) is better than in Schnyder's drawing. 
On the examples we have tested, our algorithm seems to output aesthetically pleasant drawings; see Figure~\ref{fig:5connbig_drawing} for an example.
%the rendering on the examples we have tested is aesthetically pleasant, and as we will show, the vertices are better kept distant from each other than in Schnyder's drawing.

Our article is organized as follows.
All the combinatorial results are presented in Section~\ref{sec:combin}. 
\llncs{We start by defining the 5c-structures for triangulations of the pentagon: we give three different incarnations of these structures (5c-woods, 5c-labelings, 5c-orientations) and show their equivalence; see Figure~\ref{fig:5c_structures} for an example (the three incarnations parallel those for Schnyder structures shown in Figure~\ref{fig:triangulation2}). 
%We then give a necessary and sufficient condition (closely related to 5-connectedness) for a triangulation of the pentagon to admit a Schnyder structure. 
We then show that a triangulation of the pentagon admit a 5c-wood if and only if its cycles of length less than 5 have no vertex in their interior, and we provide a linear algorithm for constructing a 5c-wood in this case.}
\addllncs{We start by defining the 5c-structures for triangulations of the pentagon (see Figure~\ref{fig:5c_structures} for an example, which parallels Figure~\ref{fig:triangulation2}). We give three different incarnations of these structures (5c-woods, 5c-labelings, 5c-orientations), we explain the equivalence between them, state the exact condition for their existence,} \addllncswa{and point to the appendix for detailed proofs and the description of a linear-time construction algorithm.}\addllncswoa{and point to the appendix of the full version~\cite{OB-EF-SL:5c-Schnyder} for detailed proofs and the description of a linear-time construction algorithm.}
The drawing algorithm is then presented in Section~\ref{sec:algo}. Together with the proof of planarity of the drawing\llncs{, and of the linear time complexity}, this section includes a discussion of the drawing properties mentioned above, and some open questions.

Let us mention that the 5c-structures presented here are closely related to the \emph{quasi-Schnyder structures} 
%defined in~\cite{OB-EF-SL:Grand-Schnyder}. 
which we introduced in~\cite{OB-EF-SL:Grand-Schnyder}.  
The quasi-Schnyder structures are a far-reaching generalization of Schnyder woods (encompassing regular edge labelings~\cite{He93:reg-edge-labeling},
 and Schnyder decompositions considered in~\cite{OB-EF:Schnyder}), 
and in the case of triangulations of the pentagon, these structures can be identified with the 5c-structures which are the focus of the present article. Focusing on triangulations of the pentagon allows us to provide a simplified presentation, both in terms of definitions and proofs, and also to define an additional incarnation in terms of woods. Let us add that the quasi-Schnyder structures for triangulations of the square coincide (after simplifications) with \emph{regular edge labelings}~\cite{He93:reg-edge-labeling} (a.k.a. \emph{transversal structures}~\cite{Fu07b}), and the quasi-Schnyder structures for triangulations \llncs{(of the triangle) }coincide (after simplifications) with the classical Schnyder woods~\cite{Schnyder:wood1}. 
%% Added. 
Another combinatorial structure which bears some resemblance to 5c-structures are the five color forests introduced in~\cite{FeScSt18} that are related
to pentagon contact representations. As we will explain in Remark~\ref{rk:pentagons}, one can easily construct a  five color forest from a 5c-structure (but the opposite is not true). 
%Other structures which we can relate to 5c-structures  are the flow-orientations introduced in~\cite{FeScSt18} that are associated to contact-representations of pentagons (more details on this will be given in Remark~\ref{rk:pentagons}).

\fig{width=\linewidth}{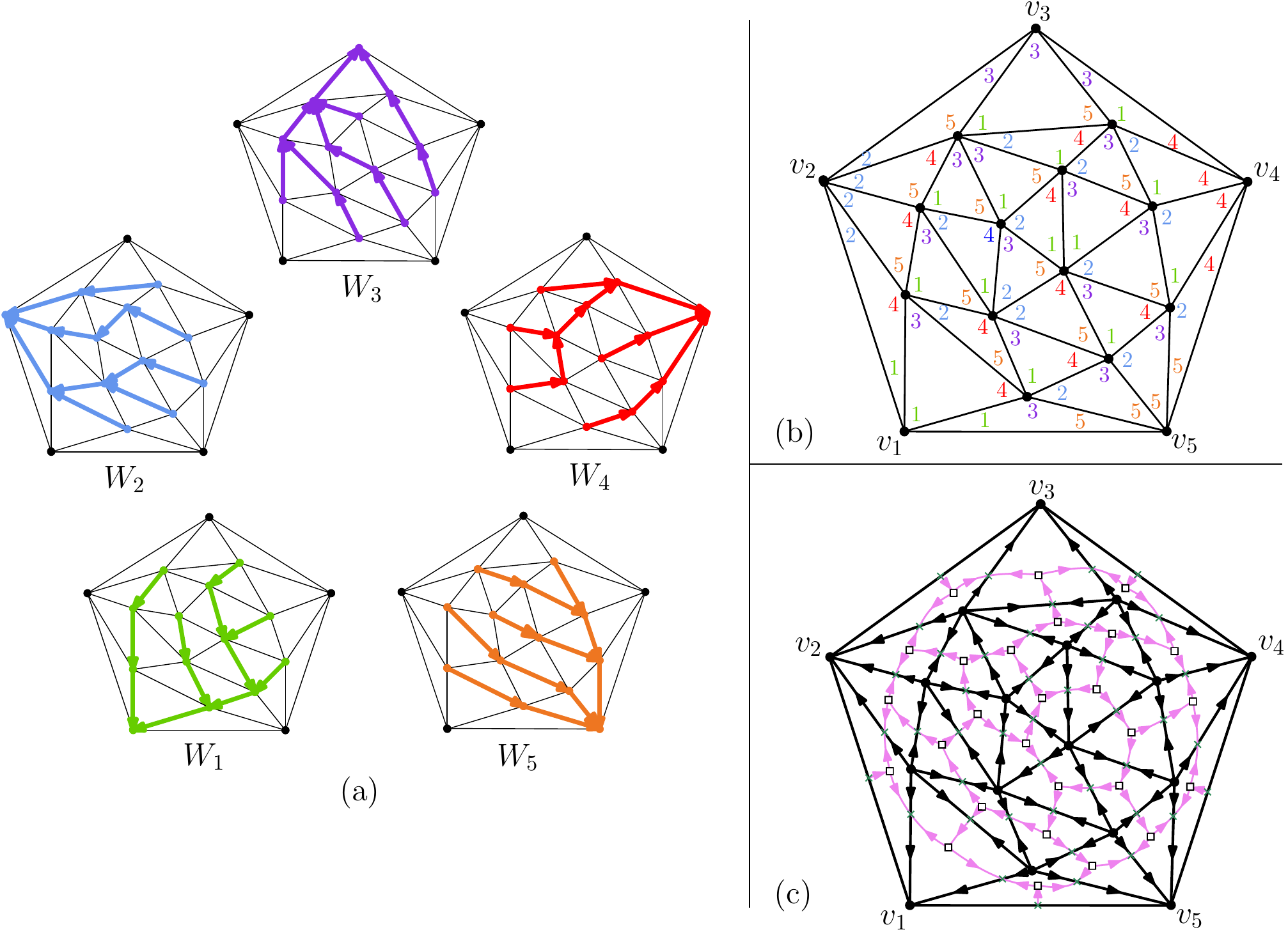}{The three incarnations of \addllncs{5c-structures.}\llncs{Schnyder-like structures for triangulations of the pentagon.} 
(a) A 5c-wood $\cW=(W_1,\ldots,W_5)$. (b) The 5c-labeling \llncs{$\cL=\ov \Theta(\cW)$}\addllncs{$\cL=\Theta^{-1}\!(\cW)$}. (c) The 5c-orientation $\cO=\Phi(\cL)$.}

%% Lastly, in Section~\ref{sec:level2-GS} we will consider an extension of grand-Schnyder woods, called \emph{quasi-Schnyder structures}, which exist for $d$-maps which are not quite $d$-adapted. 
%% We say that a map is \emph{quasi $d$-adapted} if it is a $d$-map such that simple cycles of length less than $d$ are either facial cycles, or cycles of length $d-1$ containing a single edge and no vertex. \\ 
%% We show that a $d$-map with no face of degree less than $3$ admits a quasi-Schnyder wood if and only if it is quasi $d$-adapted. 
%% A quasi-Schnyder wood of a (quasi 5-adapted) triangulation of a 5-gon is represented in Figure~\ref{fig:5c_wood}. 
%% In this figure, the quasi-Schnyder structure is represented in terms of woods, whereas the other incarnations (in terms of orientations and labelings) 
%% will be discussed in Section~\ref{sec:level2-GS}. In the upcoming article \cite{OB-EF-SL:5QS-drawing} we will focus more exclusively on quasi-Schnyder woods of quasi $5$-adapted triangulations and use these structures to define a graph-drawing algorithm (for triangulations of the 5-gon, the quasi 5-adapted condition is closely related to 5-connectedness).\\

%\bibliography{biblio-Schnyder}
%\end{document}

\section{Schnyder structures for triangulations of the pentagon}\label{sec:combin}

\subsection{Definitions about triangulations}
\llncs{A \emph{plane map} is a connected graph $G$ embedded in the plane without edge crossing.} 
\addllncs{A \emph{plane map} is a connected planar graph $G$ with a fixed planar embedding.}
The \emph{faces} of $G$ are the connected components of $\mathbb{R}^2\backslash G$. The \emph{outer face} is the unique unbounded face, the other faces are called \emph{inner}. An \emph{arc} of $G$ is an edge with a choice of direction.
We use the notation $\{u,v\}$ for an edge connecting the vertices $u$ and $v$, and $(u,v)$ for an arc oriented from $u$ to $v$. The two arcs on an edge are called \emph{opposite}. 
 A \emph{corner} is a sector delimited by two consecutive edges around a vertex. Corners, vertices, edges and arcs are called \emph{outer} when they 
are incident to the outer face, and \emph{inner} otherwise.  
%The \emph{degree} of a vertex or face is the number of corners incident to it. The edges, vertices and corners of $G$ are called \emph{outer} if they are incident to the outer face, and \emph{inner} otherwise.} 
%The edges and vertices of $G$ are called \emph{outer} if they are incident to the outer face, and \emph{inner} otherwise.
%For $d\geq 3$ a \emph{$d$-map} is a plane map $G$ where the outer face contour is a simple cycle of length $d$. The outer vertices in clockwise order around the outer face contour are denoted $v_1,\ldots,v_d$. 

A \emph{triangulation of the pentagon}, or \emph{5-triangulation} for short, is a plane map such that the inner faces have degree 3 and the outer face contour is a simple cycle of length 5. The outer vertices of a 5-triangulation are denoted by $v_1,v_2,\ldots,v_5$ in clockwise order around the outer face; see Figure~\ref{fig:5c_structures}.  
A \emph{5c-triangulation} is a 5-triangulation such that every cycle with at least one vertex in its interior has length at least 5.
%separating cycle (cycle enclosing at least one vertex) 
%\SL{(We used this term "separating cycle" only once throughout the article. In other places we've been using "cycles containing a vertex in the interior" or something similar to what's in the parenthesis. Not sure if it is necessary to use this new terminology.)} has length at least $5$. 
Note that 5c-triangulations have no loops nor multiple edges. In the terminology of~\cite{OB-EF-SL:Grand-Schnyder}, 5c-triangulations are the 5-triangulations which are \emph{quasi 5-adapted}. It is easy to check that a 5-triangulation $G$ is a 5c-triangulation if and only if the graph $G'$ obtained from $G$ by adding the edges $\{v_i,v_{i+2}\}$ for all $i\in \{1,2,\ldots,5\}$ is 5-connected (in particular 5-connected 5-triangulations are 5c-triangulations). Let us mention lastly that deleting a vertex of degree 5 from a 5-connected planar triangulation (such a vertex necessarily exists, by the Euler relation) yields a 5c-triangulation; hence the algorithm presented in Section~\ref{sec:algo} for 5c-triangulations gives a way to draw 5-connected planar triangulations (upon seeing the deleted vertex as a vertex at infinity). 
The \emph{primal-dual completion} of a plane map $G$ is the map $G^+$ obtained by inserting a vertex $v_e$ in the middle of each edge $e$, inserting a vertex $v_f$ in each inner face $f$, and then connecting $v_f$ to all edge-vertices corresponding to the edges incident to $f$. An example is shown in Figure~\ref{fig:5c_structures}(c). The vertices of $G^+$ corresponding to faces (resp. edges) of $G$ are called \emph{dual vertices} (resp. \emph{edge-vertices}), while the original vertices of $G$ are called \emph{primal vertices}.

\subsection{Three incarnations of Schnyder structures on 5-triangulations}\label{sec:incarnations}
%\documentclass[letter]{amsart}
%\input{defs.tex}
%\author{Olivier Bernardi$^{*}$ \and \'{E}ric Fusy$^{\dagger}$ \and Shizhe Liang$^{+}$}
%\title[Grand Schnyder Woods]{Grand Schnyder Woods}
%\begin{document}
%\subsection{Three incarnations of 5-quasi-Schnyder structures}\label{sec:incarnations}

%The \emph{primal-dual completion} of $G$ is the map obtained by superimposing $G$ and $G^*$, and a corner labeling of $G$. % We now assume that $G$ has at least $2$ inner vertices, so that the last property stated in Lemma~\ref{lem:rigid_qs} holds, hence the situation in the vicinity of outer edges is as shown in the left part of Figure~\ref{fig:5c_outer}(a).

%\EF{I am wondering (also due to limited space, and to have 5c-triangulations appear already in the beginning of the section) if the structure of the section could not be: definitions of 3 incarnations (with lemma for wood), and then a theorem stating that the 3 sets are in bijection, and not empty iff the 5-triangulation
%is a 5c-triangulation, and inside proof say it is essentially proved in long paper, except for the wood part, but we recall bijection orientation-labeling, and give bijection labeling-wood, and will explain how existence proof works in next subsection}

In this section we present three different incarnations of Schnyder-type structures on 5-triangulations and define bijections between \addllncs{them}\llncs{these incarnations}. The incarnations are called \emph{5c-woods}, \emph{5c-labelings} and \emph{5c-orientations} respectively. The conditions defining these structures are indicated in Figure~\ref
{fig:def-incarnations}. We start by discussing 5c-orientations, an example of which is presented in Figure~\ref{fig:5c_structures}(c).

% Definition: orient
\begin{Def}\label{def:orient}
Given a 5-triangulation $G$, a \emph{5c-orientation} of $G$ is an orientation of the inner edges of the primal-dual completion $G^+$ of $G$ \llncs{such that the following conditions are satisfied}\addllncs{satisfying the following conditions}:
\begin{enumerate}
\item[(O0)]
The outer primal vertices have outdegree~$0$.
\item[(O1)]
The inner primal vertices have outdegree~$5$, the dual vertices have outdegree~$2$, and the edge-vertices (including the outer ones) have outdegree $1$.
\end{enumerate}
\end{Def}

The definition of 5c-orientations is illustrated in the top row of Figure~\ref{fig:def-incarnations}.

\fig{width=\linewidth}{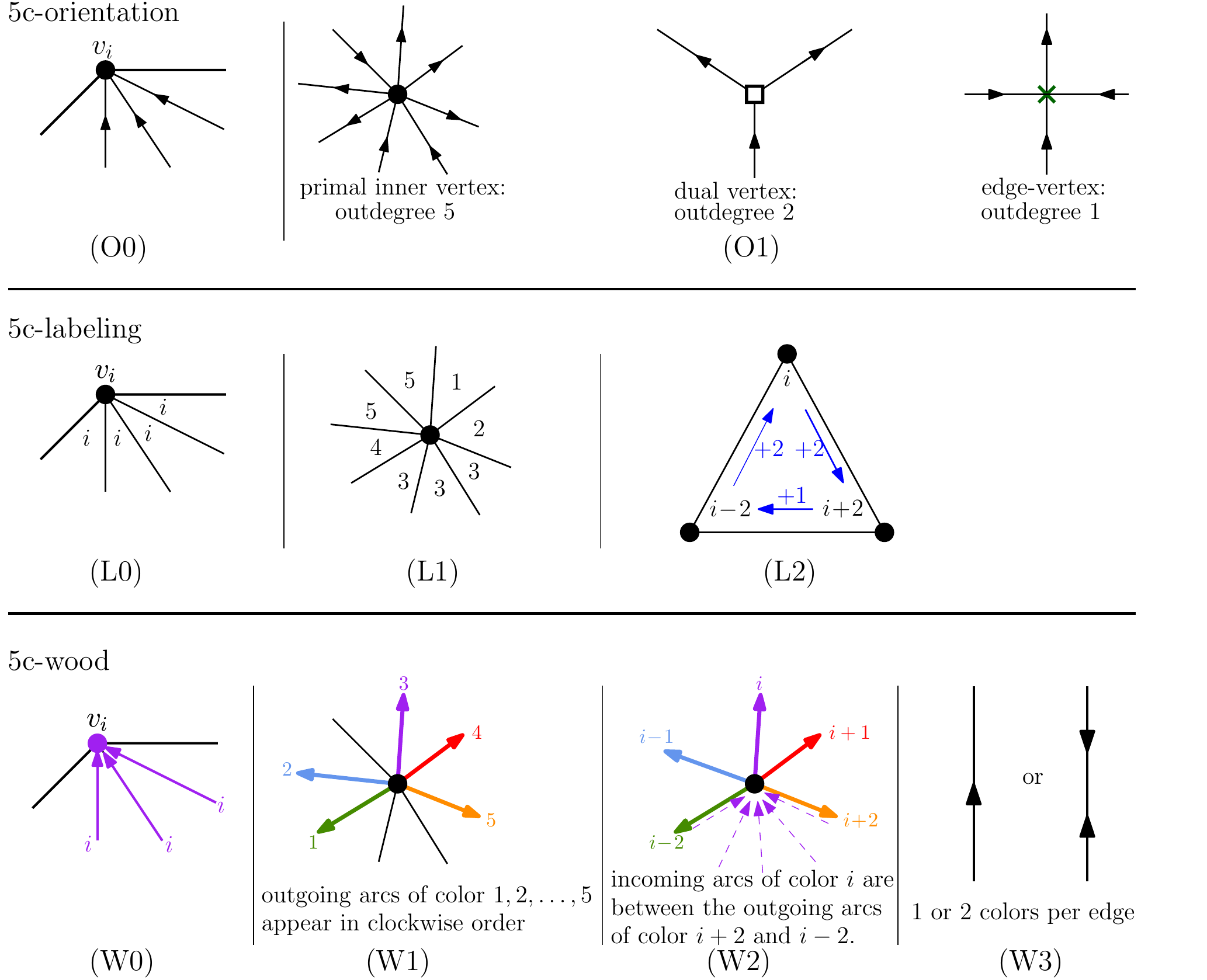}{Definition of 5c-structures. Top row: Conditions defining 5c-orientations. Middle row: Conditions defining 5c-labelings. Bottom row: Conditions defining 5c-woods.}% \SL{Should we capitalize the words "conditions" after "Top row:" and "Bottom row:"?}\OB{Yes. Thanks.}} 

% Definition: label
Next, we define \emph{5c-labelings}, an example of which is presented in Figure~\ref{fig:5c_structures}(b). A \emph{corner labeling} of a 5-triangulation $G$ is an assignment of a \emph{label} in $[1:5]:=\{1,2,3,4,5\}$ to each inner corner of $G$. 
%OB: Moved to a different place: "Also, note that there is a one-to-one correspondence between the inner corners of $G$ and the inner faces of $G^+$, hence a corner labeling of $G$ can be thought of as a labeling of the inner faces of $G^+$ by labels in $[1:5]$." 
%Given a corner labeling, and 
For two corners $c$ and $c'$ with labels $i$ and $i'$ respectively, we call \emph{label jump} from $c$ to $c'$ the integer $\de \in \{0,1,2,3,4\}$ such that $i + \de \equiv i'\pmod 5$.

\begin{Def}\label{def:label}
Given a 5-triangulation $G$, a \emph{5c-labeling} of $G$ is a corner labeling of $G$ satisfying the following conditions:
%is an assignment to each inner corner of a label in $[1:5]$  \llncs{such that the following conditions are satisfied}\addllncs{satisfying the following conditions}:
\begin{enumerate}
\item[(L0)]
For all $i\in[1:5]$, every inner corner incident to $v_i$ has label $i$.
\item[(L1)]
Around every inner vertex, the incident corners form 5 non-empty intervals $I_1,I_2,I_3,I_4,I_5$ in clockwise order, with all corners in $I_i$ having label $i$ \llncs{for $i\in[1:5]$}. 
\item[(L2)]
Around every inner face, in clockwise order, there are two label jumps equal to $2$ and one label jump equal to $1$. 
\end{enumerate}
\end{Def}

Conditions (L0-L2) are illustrated in Figure~\ref{fig:def-incarnations} (middle row).
\llncs{Note that Conditions (L1) and (L2) can be reformulated as follows. 
\begin{compactenum}
\item[1.] For every inner vertex $v$ each label jump between consecutive corners in clockwise order around $v$ is is most 1, and their sum is 5.
\item[2.] For every inner face $f$ each label jump between consecutive corners in clockwise order around $f$ is at most 2, and their sum is 5.
\end{compactenum}
}

\llncs{
We now prove an additional property of 5c-labelings.

% Property: sum of label jumps
\begin{lem}\label{lem:sum_jumps}
 Let $\cL$ be a 5c-labeling of a 5-triangulation $G$. For any inner edge $e$, the sum of label jumps between consecutive corners in counterclockwise order around $e$ is equal to 5. This is represented in Figure~\ref{fig:jumps_around_edge}.
%consider the four corners incident to $e$ and the four label jumps in counterclockwise order around $e$ as represented in Figure~\ref{fig:jumps_around_edge}. The sum of labels jumps counterclockwise around $e$ is equal to 5.
\end{lem}

\begin{proof}
\llncs{As noted above,}\addllncs{Note that} for any inner vertex or any inner face $x$, the sum $\cwjump(x)$ of label jumps between consecutive corners in clockwise order around $x$ is equal to 5. For an inner edge~$e$ we denote by $\ccwjump(e)$ the sum of label jumps between consecutive corners in counterclockwise order around~$e$. 

% Let $\cwjump(v)$ (resp. $\cwjump(f)$) denote the sum of label jumps between consecutive corners in clockwise order around an inner vertex $v$ (resp. inner face $f$), and let $\ccwjump(e)$ denote the sum of label jumps in counterclockwise order around an inner edge. By conditions (L1) and (L2), $\cwjump(v) =5$ for every inner vertex $v$ and $\cwjump(f) = 5$ for every inner face $f$.

Observe that a clockwise label jump around an inner vertex or an inner face is either a counterclockwise label jump around an inner edge, or a label jump along an outer edge (and there are exactly 5 such jumps, each with value $1$). Hence, 
$\ds \sum_{f \in F} \cwjump(f) + \sum_{v \in V} \cwjump(v) = 5 + \sum_{e \in E} \ccwjump(e),$
%\begin{equation*}
%  \sum_{f \in F} \cwjump(f) + \sum_{v \in V} \cwjump(v) = 5 + \sum_{e \in E} \ccwjump(e),
%\end{equation*} 
where $V, E, F$ are the sets of inner vertices, inner faces and inner edges, respectively. The left-hand side is equal to $5(|V|+|F|)$, which in turn is equal to $5+5|E|$ by the Euler relation. Since $\ccwjump(e)$ must be a nonzero (by Condition (L2)) multiple of 5, it has to be exactly 5, for all $e \in E$.
\end{proof}

\fig{width=0.3\linewidth}{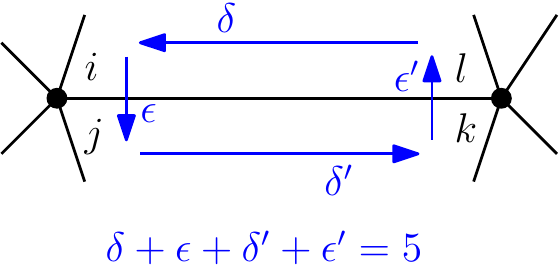}{Counterclockwise label jumps around an edge.}
}

% Bijection labeling <-> orient
\llncs{
\fig{width=\linewidth}{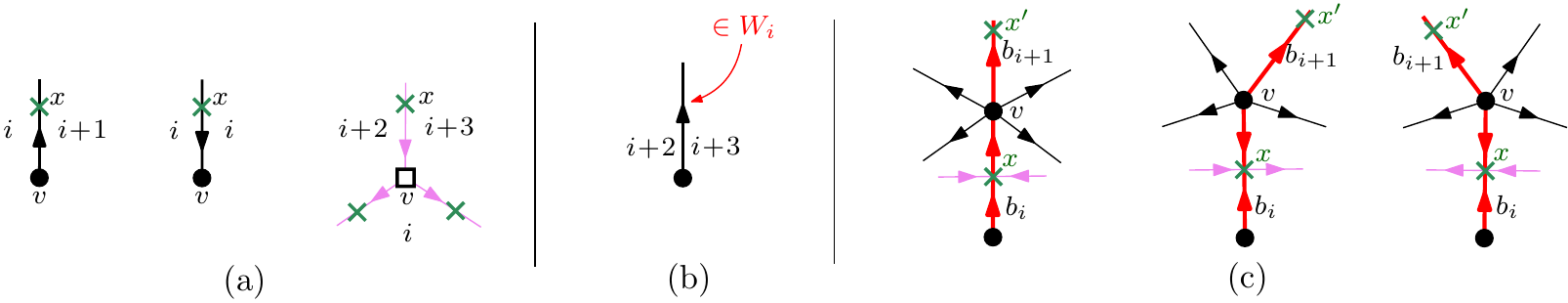}{Local rules giving the bijections between 5c-orientations, 5c-labelings, and 5c-wood. Vertices of $G$ are represented by solid black dots, whereas dual vertices (in $G^+$) are represented by squares, and edge-vertices are represented by crosses.
(a) Local rule giving the bijection $\Phi$ from 5c-labelings to 5c-orientations. (b) Local rule giving the bijection $\Theta$ from 5c-labelings to 5c-woods. (c) The straight-path rule giving the bijection $\Theta\circ \ov \Phi$ from 5c-orientations to 5c-woods.}
}
\addllncs{
\fig{width=.8\linewidth}{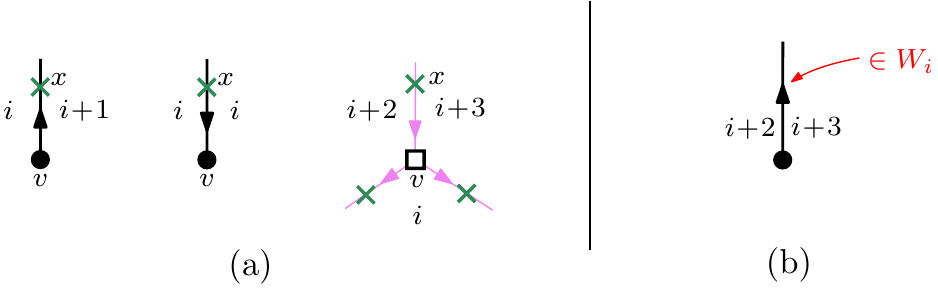}{Rules giving the bijections between 5c-orientations, 5c-labelings, and 5c-wood. Vertices of $G$ are represented by solid black dots, whereas dual vertices \llncs{(in $G^+$) }are represented by squares, and edge-vertices are represented by crosses.
(a) Local rule giving the bijection $\Phi$ from 5c-labelings to 5c-orientations. (b) Local rule giving the bijection $\Theta$ from 5c-labelings to 5c-woods.}
}

Next, we define a bijection $\Phi$ between 5c-labelings and 5c-orientations. The mapping $\Phi$ is represented in Figure~\llncs{\ref{fig:local_rules_5c}(a)}\addllncs{\ref{fig:local_rules_5c_llncs}(a)}. %Primal vertices are represented by solid black dots, whereas dual vertices are represented by squares, and edge-vertices are represented by crosses.

\begin{Def}\label{def:bij_label_orient}
Given a 5c-labeling $\cL$ of $G$, we define an orientation $\Phi(\cL)$ on $G^+$ as follows. 
First we note that there is a one-to-one correspondence between the inner corners of $G$ and the inner faces of $G^+$, hence we interpret $\cL$ as a labeling of the inner faces of $G^+$\llncs{ by labels in $[1:5]$}.
%Also, note that there is a one-to-one correspondence between the inner corners of $G$ and the inner faces of $G^+$, hence a 5c-labeling of $G$ can be thought of as a labeling of the inner faces of $G^+$ by labels in $[1:5]$.
 Let $e = \{v, x\}$ be an inner edge of $G^+$, where $v$ is either a primal or dual vertex, and $x$ is an edge-vertex. 
\begin{itemize}
 \item 
 If $v$ is a primal vertex, and $f^-$ and $f^+$ are the \llncs{two }faces incident to $e$ in $G^+$ in clockwise order around $v$, then $e$ is oriented toward $v$ if and only if the label jump from $f^-$ to $f^+$ is~0 (i.e., $f^-$ and $f^+$ have the same label).
 \item 
 If $v$ is a dual vertex, and $f^-$ and $f^+$ are the \llncs{two }faces incident to $e$ in $G^+$ in clockwise order around $v$, then $e$ is oriented toward $v$ if and only if the label jump from $f^-$ to $f^+$ is~1.
\end{itemize}
\end{Def}

\begin{lem}\label{lem:bij_label_orient}
The map $\Phi$ is a bijection between the set $\bL_G$ of 5c-labelings on $G$ and the set $\bO_G$ of 5c-orientations on $G$.
\end{lem}
\addllncswa{The proof of Lemma \ref{lem:bij_label_orient} can be found in Appendix~\ref{sec:pf_lem_bij_label_orien}.}
\addllncswoa{The proof of Lemma \ref{lem:bij_label_orient} can be found in the appendix of~\cite{OB-EF-SL:5c-Schnyder}.}
\llncs{
\begin{proof}
Let $\cL\in\bL_G$. 
We first show that $\Phi(\cL)$ is a 5c-orientation. Condition (O0) is a direct consequence of Condition (L0). Every inner primal vertex has outdegree 5 by Condition (L1), and every dual vertex has outdegree 2 by Condition (L2). 
Also, Lemma~\ref{lem:sum_jumps} implies that every edge-vertex has outdegree 1. Hence, Condition (O1) holds, and  $\Phi(\cL)$ is a 5c-orientation. 

%We prove $\Phi$ is a bijection by providing its inverse map. 
We now describe the inverse map. 
Let $\mathcal{O}$ be a 5c-orientation. The orientations of edges in $\cO$ dictate the label jumps between consecutive corners around vertices and faces. 
One can then follow these label jumps in order to propagate labels, starting from corners at the outer vertices (which are fixed by Condition (L0)) inward to all corners. If no conflict arises in this propagation (so that all the dictated label jumps are satisfied), then one obtain a labeling of corners $\cL$, which we denote by $\ov{\Phi}(\cO)$. We prove below that, for any 5c-orientation $\cO$, no conflict occurs during the propagation of labels, and that $\ov{\Phi}(\cO)$ is a 5c-labeling.

% let $\cL$ be the resulting labeling of corners, and let $\ov{\Phi}$ be the mapping associating $\cL$ to $\cO$. 
%To show that $\cL$ (and thus $\ov{\Phi}$) is well-defined, we need to prove that no conflicts occurs during the propagation of labels.

Consider the \emph{corner graph} $C_G$ of $G$ defined as follows: $C_G$ is a directed graph whose vertices are the inner corners of $G$, and there is an oriented edge from a corner $c$ to a corner $c'$ in $C_G$ if $c'$ is the corner following $c$ in clockwise order around a face or a vertex. A corner graph is represented in Figure~\ref{fig:corner_graph}. $C_G$ is naturally endowed with an embedding which is induced by the embedding of $G$ (in fact $C_G$ is obtained from the dual of $G^+$ by deleting a vertex). In particular, $C_G$ has three types of inner faces, corresponding to inner vertices, inner edges and inner faces of $G$ respectively.

We define the $\cO$\emph{-weight} $w(a)$ of an arc $a = (c,c')$ of $C_G$ to be the label jump from $c$ to $c'$ determined by $\mathcal{O}$ according to the rules given by Figure~\llncs{\ref{fig:local_rules_5c}(a)}\addllncs{\ref{fig:local_rules_5c_llncs}(a)}. The $\cO$\emph{-weight} $w(P)$ of a directed path $P$ of $C_G$ is the sum of the $\cO$-weight of the arcs of $P$. Note that the propagation rule causes no conflicts if and only if for any two (not necessarily consecutive) corners $c$ and $c'$ and any two directed paths $P_1, P_2$ in $C_G$ from $c$ to $c'$, we have $w(P_1) \equiv w(P_2) \pmod 5$. To show the later property it suffices to show that for any simple cycle $C$ in $C_G$, we have 
$$\sum_{a \in C^+}w(a) - \sum_{a \in C^-}w(a) \equiv 0 \pmod 5,$$ 
where $C^+$ and $C^-$ are the sets of arcs appearing clockwise and counterclockwise on $C$, respectively. It is easy to see that this holds if and only if the $\cO$-weight of the contour of each inner face in $C_G$ is congruent to 0 modulo 5. This last condition is implied by Condition (O1) of 5c-orientations.  
%Condition (O1) of 5c-orientations ensures that the $\cO$-weight of the contour of each inner face in $C_G$ is congruent to 0 modulo 5 \SL{(what about: but this is ensured by Condition (O1) of 5c-orientations?)}. 
Hence the corner labeling $\cL=\ov{\Phi}(\cO)$ is well-defined. Moreover $\cL$ is indeed a 5c-labeling because Condition (L0) is satisfied by definition, and Conditions (L1) and (L2) are easy consequences of Condition (O1). 

%% Then Condition (O1) implies that the weight of the contour of each type of inner face in $C_G$ is equal to 5.

%% Note that the propagation rule causes no conflicts if and only if for any two (not necessarily consecutive) corners $c$ and $c'$ and any two directed paths $P_1, P_2$ in $C_G$ from $c$ to $c'$, we have $w(P_1) \equiv w(P_2) \pmod 5$. To show the later condition it suffices to show that for any simple cycle $C$ in $C_G$, we have $$\sum_{a \in C^+}w(a) - \sum_{a \in C^-}w(a) \equiv 0 \pmod 5,$$ where $C^+$ and $C^-$ are the sets of directed edges appearing clockwise and counterclockwise on $C$, respectively. It is easy to see that this holds if and only if the weight of the contour of each inner face in $C_G$ is congruent to 0 modulo 5. Hence $\ov{\Phi}$ is well-defined. Also, the resulting corner labeling $\cL$ is indeed a 5c-labeling because Condition (L0) is satisfied by definition of $\ov{\Phi}$ and (L1) and (L2) are easy consequences of Condition (O1). 

Finally, it is easy to see that $\Phi$ and $\ov{\Phi}$ are inverse to each other, hence they give bijections between 5c-labelings and 5c-orientations of $G$.
\end{proof}

\fig{width=0.8\linewidth}{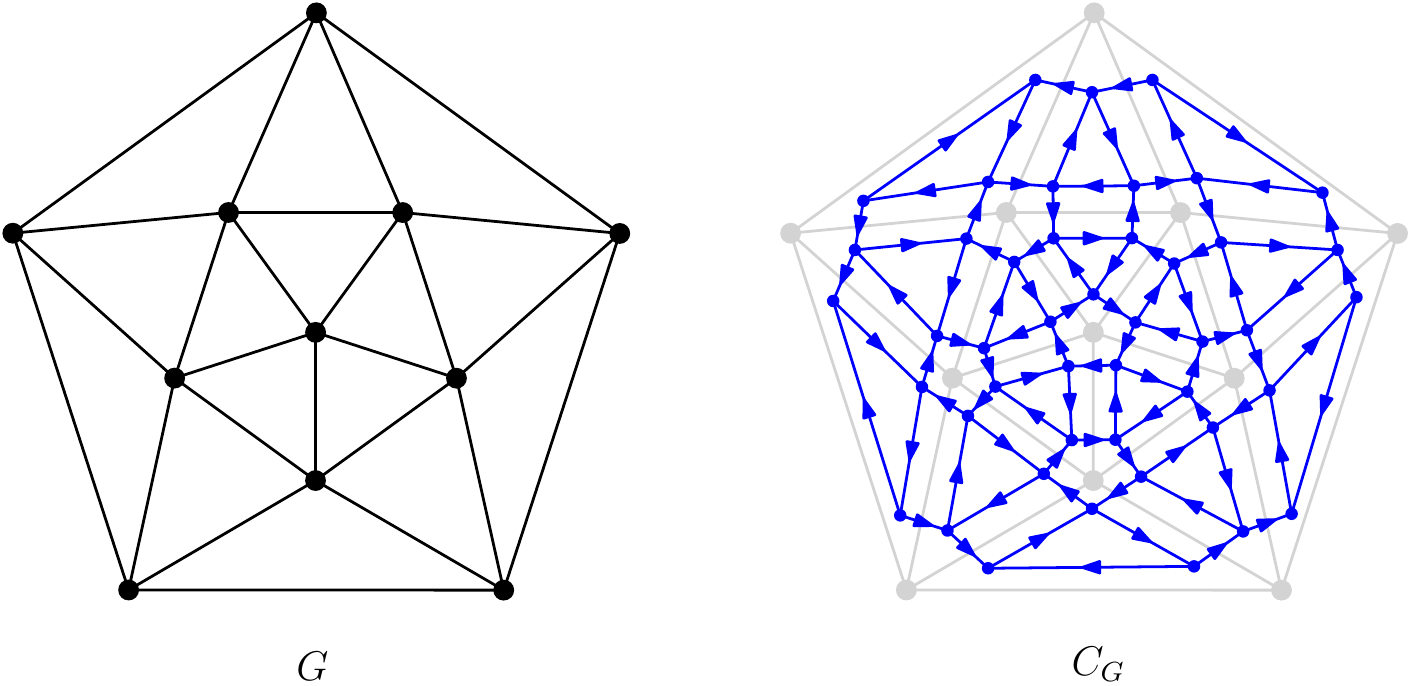}{A 5-triangulation $G$ and its corner graph $C_G$.}

}

% Definition: wood
The third incarnation \llncs{of 5c-structures is as woods}\addllncs{is called \emph{5c-woods}}, where the structure is encoded by a tuple $\cW = (W_1,...,W_5)$ of directed trees. An example is given in Figure~\ref{fig:5c_structures}(a).

It is sometimes convenient to think of the tuple $\cW = (W_1,...,W_5)$ as a coloring in $[1:5]$ of the inner arcs of $G$. We say that an inner arc $a$ of $G$ has a color \llncs{$i \in [1:5]$}\addllncs{$i$} if this arc belongs to\llncs{the set} $W_i$. 
The precise definition of a 5c-wood is as follows.

\begin{Def}\label{def:wood}
Given a 5-triangulation $G$, a \emph{5c-wood} of $G$ is a tuple $\cW = (W_1,...,W_5)$ of disjoint subsets of inner arcs, satisfying the following conditions:
\begin{enumerate}
  \item[(W0)] No arc in $\cW$ starts at an outer vertex, and those ending at the outer vertex $v_i$ have color~$i$ for all $i\in[1:5]$. 
	\item[(W1)] Every inner vertex $v$ has a unique outgoing arc of color $i$, for \llncs{each} $i \in [1:5]$, and these arcs appear in clockwise order around $v$.
	\item[(W2)] Let $v$ be an inner vertex with incident outgoing arcs $a_1,...,a_5$ of colors $1,...,5$, respectively.  Any arc $a$ of color $i$ having terminal vertex $v$ appears weakly between $a_{i+2}$ and $a_{i+3}$ in clockwise order around $v$ (\emph{weakly} means that $a$ may be the arc opposite to $a_{i+2}$ or $a_{i+3}$).
%Let $a$ be an arc oriented toward $v$. Suppose $a$ has color $i$, then $a$ appears weakly between $a_{i+2}$ and $a_{i+3}$ in clockwise order around $v$.
%	\item[(W3) Every inner edge has at least one color; those with one color (resp. two colors) are called uni-directed (resp. bi-directed).
 	\item[(W3)] Every inner edge has at least one color.
\end{enumerate}
\end{Def}

%The wood incarnation is represented in Figure~\ref{fig:5c_wood}. 
Conditions (W0-W3) are illustrated in the bottom row of Figure~\ref{fig:def-incarnations}.
We will prove later (see Proposition~\ref{prop:acyclic}) that for any 5c-wood $\cW=(W_1,\ldots,W_5)$ the set of arcs $W_i$ is acyclic for all~$i$. Given that every inner vertex has outdegree 1 in $W_i$, it follows that \llncs{the set of arcs}$W_i$ is a tree spanning all the inner vertices \llncs{together with}\addllncs{and} the outer vertex $v_i$, and oriented toward the root-vertex $v_i$ (that is, every arc in $W_i$ is oriented from child to parent when $v_i$ is taken as the root of the tree $W_i$).

% Bijection: label <-> wood
Now we define a bijection $\Theta$ between 5c-labelings and 5c-woods. The mapping $\Theta$ is represented in Figure~\llncs{\ref{fig:local_rules_5c}(b)}\addllncs{\ref{fig:local_rules_5c_llncs}(b)}.

\begin{Def}\label{def:bij_label_wood}
Given a 5c-labeling $\cL$ of $G$, we define a tuple $\Theta(\cL) = (W_1,...,W_5)$ of subsets of inner arcs of $G$ (interpreted as a partial arc coloring) as follows: an inner arc $a=(u,v)$  receives color $i$ if the labels of the corners on the left and on the right of $a$ (at $u$) are $i+2$ and $i+3$, respectively; the arc $a$ has no color if the two labels are equal.
\end{Def}

\begin{lem}\label{lem:bij_label_wood}
The map $\Theta$ is a bijection between the set $\bL_G$ of 5c-labelings on $G$ and the set $\bW_G$ of 5c-woods on $G$.
\end{lem}

\addllncswa{The proof of Lemma \ref{lem:bij_label_wood} can be found in Appendix~\ref{sec:lem_bij_label_wood}.}
\addllncswoa{The proof of Lemma \ref{lem:bij_label_wood} can be found in the appendix of~\cite{OB-EF-SL:5c-Schnyder}.}

\llncs{
\begin{proof}
 We first show that $\Theta(\cL)$ is a 5c-wood. By (L0), the arcs that start at outer vertices receive no color. 
 %For any $i \in [1:5]$, consider the inner faces incident to the outer vertex $v_i$ in counterclockwise order around $v_i$. Focus on the face incident to the edge $\{v_{i-1},v_i\}$ first. By (L0) the corner at $v_{i-1}$ has label $i-1$ and the corner at $v_i$ has label $i$, then by (L2) the other corner has label $i+2$. It is then easy to see that Conditions (L1)-(L2) force all the subsequent faces to have labels $\{i, i+2, i+3\}$, except for the last one which is incident to the outer edge $\{v_i, v_{i+1}\}$ (this face has labels $\{i,i+1,i+3\}$). See Figure~\ref{fig:outer_labels}(a). 
 %This implies that all the inner arcs ending at $v_i$ receive color $i$ by $\Theta$. Hence $\Theta(\cL)$ satisfies (W0).
 Moreover, Conditions~(L1) and~(L2) and Lemma~\ref{lem:sum_jumps} imply that for any inner edge $e=\{u,v\}$ with labels $i$ on each side at $v$, the labels at $u$ must be $i+2$ 
 on the left and $i+3$ on the right, so that the arc $(u,v)$ receives label $i$. 
 This implies that the inner arcs ending at $v_i$ receive color $i$ by $\Theta$. Hence $\Theta(\cL)$ satisfies (W0).

% \fig{width=0.8\linewidth}{outer_labels}{(a) Labels at $v_i$. (b) Corner labels around an edge.}

 It is clear that (W1) is a direct consequence of (L1). To verify (W2), consider an inner arc $a=(u,v)$ of color $i\in\set5$, with $v$ an inner vertex.   
 Let $i+2, i+3, j, k$ be the labels in counterclockwise order around $a$, starting with those two incident to $u$.  %as indicated in Figure~\ref{fig:outer_labels}(b). 
 By (L2), $j \in \{i-1,i\}$ and $k \in \{i,i+1\}$. Therefore the opposite arc $-a$ has either no color while being 
 between the arcs of colors $i+2$ and $i+3$ starting from $v$, or it has color $i+2$, or color $i+3$, which proves that the arc $a$ satisfies (W2).

 Lastly, for any inner edge $e=\{u,v\}$,  
 %recall that Lemma~\ref{lem:sum_jumps} guarantees that the sum of labels jumps in counterclockwise order around an inner edge $e$ is always 5. 
 Lemma~\ref{lem:sum_jumps} and Conditions~(L1) and (L2) easily imply that either 1 or 2 arcs of $e$ receive a color (depending on having one or two  
  label jumps equal to $2$ along $e$, giving respectively two or one label jumps equal to 1 across $e$).   
 Therefore $\Theta(\cL)$ satisfies (W3) and is a 5c-wood.

 To prove that $\Theta$ is a bijection, we now describe the inverse map. Given a 5c-wood $\cW$ on $G$, we define a corner labeling $\ov{\Theta}(\cW)$ as follows:
 \begin{itemize}
\item  the inner corners incident to the outer vertex $v_i$ receive label $i$; 
\item a corner incident to an inner vertex $v$ receives label $i$ if it is between the outgoing arcs of colors $i+2$ and $i+3$ in clockwise order around $v$.
\end{itemize}

We need to show that $\ov{\Theta}(\cW)$ is indeed a 5c-labeling. Conditions (L0) and (L1) are clearly satisfied. 
Next we check that any label jump in clockwise direction around a face is equal to~1 or~2. Let $e=\{u,v\}$ be an inner edge of $G$. 
By (W3), one of its two arcs has some color $i$, say the one starting at $u$. Let $i+2, i+3, j, k$ be the labels of the corners in counterclockwise order around $e$, starting
with those two incident to $u$.  %as indicated in Figure~\ref{fig:outer_labels}(b). 
Then (W2) implies that $(j,k) \in \{(i-1,i), (i,i), (i,i+1)\}$. In all cases, the label jumps from $i+3$ to $j$ and from $k$ to $i+2$ are both in $\{1,2\}$. Since the sum of label jumps in clockwise direction around a face is a multiple of 5, and each jump is in $\{1,2\}$, the only possibility is that one label jump is equal to 1 and two label jumps are equal to 2. This shows that $\ov{\Theta}(\cW)$ satisfies (L2), hence is a 5c-wood. 

Finally, it is clear that the mappings $\Theta$ and $\ov{\Theta}$ are inverse of each other, hence they give bijections between the sets of 5c-labelings and 5c-woods of $G$. 
\end{proof}
}

\begin{remark}\label{rk:pentagons}
We can give a fourth incarnation of 5c-structures, as a representation of the 5-triangulation by a contact system of ``soft pentagons'' as indicated in Figure~\ref{fig:soft-pentagons-bis}. Let $G$ be a 5-triangulation. A \emph{soft pentagon contact representation} of $G$ is a collection of interior-disjoint ``soft pentagons'' (the sides of which are curves rather than straight-line segments) inside a regular ``outer pentagon'', with vertices of each pentagon labeled 1 to 5 in clockwise order such that
\begin{compactitem}
%\item the outer pentagon corresponds to the outer vertices and outer edges of $G$,
\item the side $[i+2,i+3]$ of the outer pentagon corresponds to the outer vertex $v_i$ of $G$, for $i\in[1:5]$,
\item the soft pentagons correspond to the inner vertices of $G$,
\item contacts between pentagons correspond to the inner edges of $G$,
\item all the pentagon contacts occur between a soft pentagon vertex of label $i\in[1:5]$ and either the side labeled $[i+2,i+3]$ of another soft pentagon (vertex-to-side contact), or the vertex $i+2$ or $i+3$ of another soft pentagon (vertex-to-vertex contact), or the side labeled $[i+2,i+3]$ of the outer pentagon (outer contact, always vertex-to-side).  
\end{compactitem}
A soft pentagon contact representation is \emph{complete} if every vertex of every soft pentagon has a contact.
It is easy to see that the 5c-woods of $G$ are in bijection with its complete soft pentagon contact representations. This is illustrated in Figure \ref{fig:soft-pentagons-bis}. 
%Conversely soft-pentagon contact representations of $G$ give a 5c-woods. Thus  5c-structures of $G$ are in bijection with the  complete soft-pentagon contact representations of $G$. 
This topological interpretation of 5c-woods allows us to compare them to the  \emph{five color forests} introduced in~\cite{FeScSt18}. For a 5-triangulation $G$ these structures correspond to general (potentially incomplete) soft pentagon contact representations without vertex-to-vertex contact. In particular, one can easily turn any 5c-wood into a five color forest, upon resolving each vertex-to-vertex contact into a vertex-to-side contact. At the level of the 5c-wood (interpreted as a collection of arcs colored in $[1:5]$), this amounts to deleting one colored arc for each edge of $G$ bearing 2 colored arcs. Note however that it is usually not possible to turn a five color forest into a 5c-wood. Indeed, five color forests exist for any simple 5-triangulation, while (as stated below) 5c-woods only exist for 5c-triangulations.
\end{remark}

\fig{width=.7\linewidth}{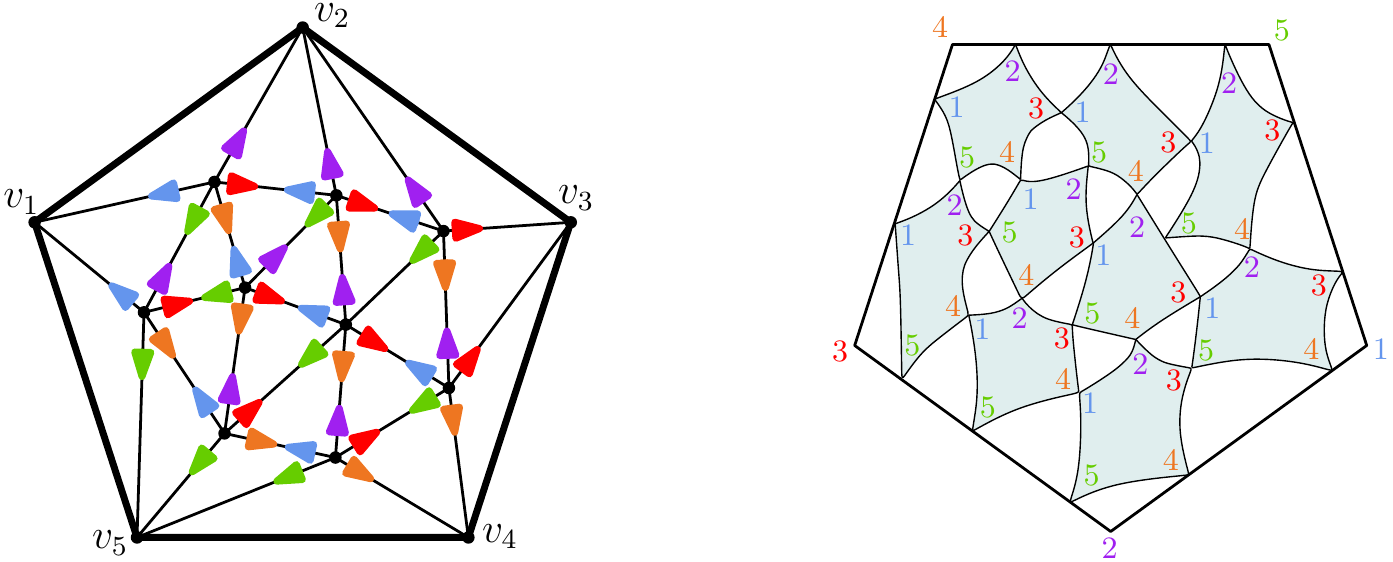}{Left: A 5c-wood. Right: The corresponding (complete) soft pentagon contact representation.}

\llncs{Before closing this section, we describe the bijection $\Psi$ between 5c-orientations and 5c-woods obtained by composing the bijections $\Theta$ and $\ov{\Phi}$ defined above.
Let $\cO$ be a 5c-orientation of a 5-triangulation $G$. We associate to $\cO$ a (partial) coloring $\Psi(\cO)$ of its inner arc as follows.
Let $a=(u,u')$ be an inner arc of $G$, and let $x$ be the edge-vertex of $G^+$ corresponding to the edge $\{u,u'\}$. If the edge $\{u,x\}$ of $G^+$ is oriented toward $u$ in $\cO$, then the arc $a$ has no color in $\Psi(\cO)$. Otherwise the arc $a$ has a color determined by the endpoint of a directed path $P_a$ starting with the arc $a$, and constructed by applying repeatedly the \emph{straight-path rule}\footnote{Our \emph{straight-path rule} is the analogue of the one described in~\cite[Fig.11]{Felsner:lattice} for 3-connected Schnyder woods.} indicated in Figure~\ref{fig:local_rules_5c}(c). Explicitly, we consider the directed path $P_a$ on $G$ ending at an outer vertex, and made of the arcs $b_1,\ldots,b_k$, where $b_1=a$ and the arc $b_{i+1}$ is determined from the arc $b_i$ for all $i\in [1:k-1]$ using the following rule. 
%\SL{(Still, we've been using the notation $b_i$ to mean two things. Not sure if this would be confusing for the reader, but the additional explanation for the index $i$ certainly makes it better.)}. 
Let $v$ be the terminal vertex of $b_i$ (which is supposed to be an inner vertex), and let $x$ be the edge-vertex of $G^+$ placed in the middle of $b_i$. 
\begin{compactitem}
\item If the edge $\{x,v\}$ is oriented toward $v$ on $\cO$, then we consider the outgoing arc $(v,x')$ of $\cO$ such that there are 2 outgoing arcs between $(v,x)$ and $(v,x')$ in clockwise order around $v$. We define $b_{i+1}$ as the arc $(v,w)$ of $G$ on which $x'$ is placed.
\item If the edge $\{x,v\}$ is oriented toward $x$ on $\cO$ and the outgoing arc at $x$ in $\cO$ is on the right (resp. left) of $b_i$, then we consider the 
  outgoing arc $(v,x')$ of $\cO$ such that there are 2 outgoing arcs between $(v,x)$ and $(v,x')$ in clockwise (resp. counterclockwise) order around $v$. We define $b_{i+1}$ as the arc $(v,w)$ of $G$ on which $x'$ is placed.
 \end{compactitem}
Finally, we consider the outer vertex $v_j$ at which the path $P_a$ ends, and assign color $j$ to the arc $a$. It is not hard to check that the arc coloring $\Psi(\cO)$ is equal to the 5c-wood $\Theta\circ\ov\Phi(\cO)$.
}

%\EF{To be mentioned that the straight-path rule is similar to the one used in the correspondence between 3-connected Schnyder woods and the related outdegree-constrained orientations, Fig.11 in~\cite[Fig.11]{Felsner:lattice}}

%\bibliography{biblio-Schnyder}
%\end{document}

\subsection{Existence and computation of Schnyder structures}\label{sec:construction}

We now state\llncs{ and prove} our main result on 5c-structures (where we include the already obtained bijective statements).

\begin{thm}\label{thm:main}
A $5$-triangulation $G$ admits a 5c-wood (resp. 5c-labeling, 5c-orientation) if and only if $G$ is a 5c-triangulation. 
In this case, the sets of 5c-woods, 5c-labelings and 5c-orientations of $G$ are in bijection. Moreover, a 5c-wood (resp. 5c-labeling, 5c-orientation) of a 5c-triangulation can be computed in  linear time in the number of vertices.
\end{thm}

We have already \llncs{shown}\addllncs{stated} in Section~\ref{sec:incarnations} that, for any $5$-triangulation $G$, the sets of 5c-woods, 5c-labelings and 5c-orientations of $G$ are in bijection. Moreover, it is clear that these bijections can be performed in linear time in the number of vertices. 
\llncs{Hence, it now suffices}\addllncs{It thus remains} to show that a 5-triangulation admits a 5c-orientation if and only if it is a 5c-triangulation, and that, in this case, a 5c-orientation can be computed in time linear.
%This existence (and algorithmic) result was already proved in \cite{OB-EF-SL:Grand-Schnyder} within the larger framework of \emph{quasi-Schnyder structures}.} 
We already proved this existence (and algorithmic) result  in \cite{OB-EF-SL:Grand-Schnyder} within the larger framework of \emph{quasi-Schnyder structures}. 
\llncs{Since there are many layers to the proof given in \cite{OB-EF-SL:Grand-Schnyder}, we sketch a more direct proof (and construction algorithm) below.}
\addllncs{Since there are many layers to the proof given in \cite{OB-EF-SL:Grand-Schnyder}, we sketch a more direct proof (and construction algorithm)}
\addllncswa{in Appendix~\ref{sec:pf_main}.}\addllncswoa{in the appendix of~\cite{OB-EF-SL:5c-Schnyder}.} 

\llncs{
Let us first show that a 5-triangulation $G$ admitting a 5c-orientation is necessarily a 5c-triangulation. If $G$ admits a 5c-orientation, it also admits a 5c-wood $\cW=(W_1,\ldots,W_5)$. For an inner vertex $v$ of $G$, we consider the paths $P_i(v)$ in $W_i$ from $v$ to $v_i$, for $i \in [1:5]$. 
As we will establish in Section~\ref{sec:paths_regions} the paths $P_1(v),\ldots,P_5(v)$ have no vertex in common except $v$.  Hence $v$ cannot be in the interior of a simple cycle of length less than 5. This proves that $G$ is a 5c-triangulation.

In the rest of this section, we describe a linear time algorithm for computing a 5c-orientation of 
%\SL{on? for?}\OB{``of'' or ``for'' both work. ``on'' does not work} 
a 5c-triangulation $G$. The process is illustrated in Figure~\ref{fig:construction}, it uses  some facts about \emph{regular edge labelings}, which are combinatorial structures for triangulations of the 4-gon first defined by He in \cite{He93:reg-edge-labeling}, and rediscovered in \cite{Fu07b} under the name of \emph{transversal structures}. We will rely on an incarnation of these structures as certain outdegree-constrained orientations.  
Let $H$ be a triangulation of the 4-gon, and let $\Hd$ be the map obtained by placing a \emph{face-vertex} in each inner face of $H$ and joining it to the three vertices incident to this face (and then erasing the original edges of $H$). 

A \emph{regular orientation} of $H$ is an orientation of $\Hd$ such that every face-vertex of $\Hd$ has outdegree 1 and every original inner vertex of $H$ has outdegree 4; an example is shown in Figure~\ref{fig:construction}(b). It follows from \cite{He93:reg-edge-labeling,Fu07b} that a triangulation of the 4-gon $H$ admits a regular orientation if and only if $H$ has no cycle of length less than 4 containing a vertex in its interior (equivalently, it has no loop nor multiple edges, and every 3-cycle is the boundary of a face). Moreover, such an orientation can be computed in linear time~\cite{KantHe97:reg-edge-labeling-linear} (an independent proof is given in \cite{OB-EF-SL:Grand-Schnyder}).

%We call \emph{regular edge labeling} of $H$ an orientation of $\Hd$ such that every face-vertex of $\Hd$ has outdegree 1 and every original inner vertex of $H$ has outdegree 4; an example is shown in Figure~\ref{fig:construction}(b). It is shown in \cite{He93:reg-edge-labeling} that a triangulation of the 4-gon $H$ admits a regular edge labeling if and only if every 3-cycle of $H$ is the boundary of a face. Moreover, it is shown in \cite{KantHe97:reg-edge-labeling-linear} that such a structure can be computed in linear time in the number of vertices (an independent proof is given in \cite{OB-EF-SL:Grand-Schnyder}). 

\fig{width=\linewidth}{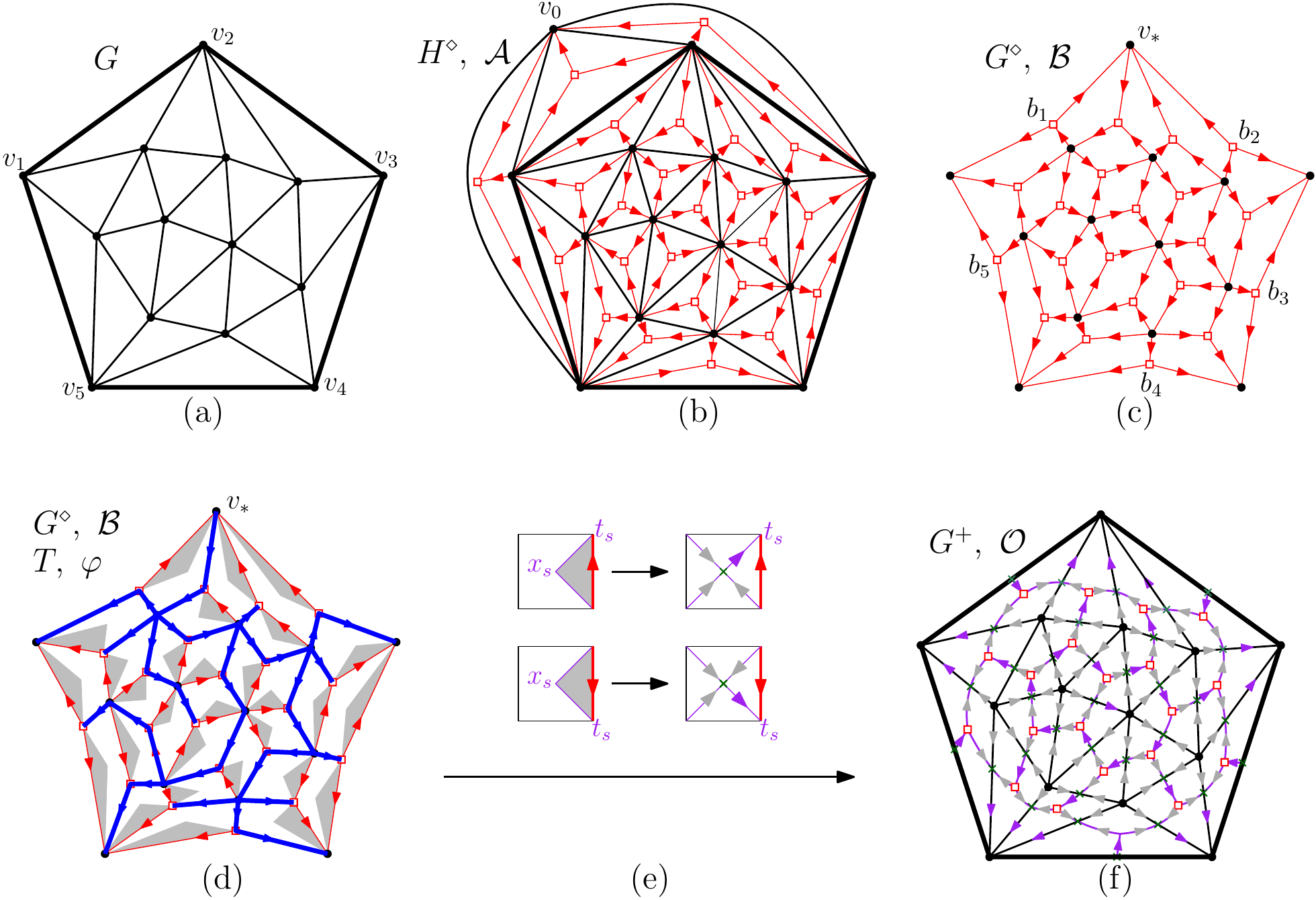}{Construction of a 5c-orientation for a 5c-triangulation. (a) A 5c-triangulation $G$. (b) The triangulation of the $4$-gon $H$ (obtained from $G$ by placing a vertex $v_0$ in the outer face of $G$), and the regular orientation $\cA$ of $\Hd$. (c) The orientation $\cB$ of $\Gd$. (d) The spanning tree $T$ (indicated by bold lines) oriented from the root vertex $v_*$ to the leaves, and the bijection $\varphi$ between the set $S$ of inner faces of $\Gd$ and set $\ov T$ of edges not in $T$: for each edge $e$ in $\ov T$, the corresponding face $\varphi(e)$ is indicated by a shaded triangle inside of $\varphi(e)$ and incident to $e$. (e) The rule (inside each inner face $s$ of $\Gd$) for computing the orientation $\cO$ of the primal-dual completion $G^+$. (f) The resulting 5c-orientation $\cO$ of $G^+$.}

Let $G$ be a 5c-triangulation. Let $H$ be the triangulation of the 4-gon obtained from $G$ by placing a new vertex $v_0$ in the outer face of $G$ and joining $v_0$ to $v_1$, $v_2$, $v_3$ and $v_4$. 
Since $G$ is a 5c-triangulation, $H$ has no cycle of length less than 4 containing a vertex in its interior. Hence, as recalled above, $H$ admits a regular orientation $\cA$, which can be computed in linear time. The situation is represented in Figure~\ref{fig:construction}(a-b).
Let $\Gd$ be the subgraph of $\Hd$ obtained by deleting the vertices and edges of $\Hd$ in the outer face of $G$. 
Consider the restriction $\cA'$ of $\cA$ to $\Gd$. Let $b_i$ be the face-vertex of $\Gd$ in the inner face of $G$ incident to the outer edge $\{v_i,v_{i+1}\}$ for $i\in [1:5]$. By definition, the face-vertex $b_i$ has outdegree 1 in $\cA'$, and we can reorient one of the edges $\{b_i,v_i\}$ or $\{b_i,v_{i+1}\}$ to obtain an orientation $\cB$ of $\Gd$ such that 
\begin{compactitem}
\item for all $i\in [1:5]$, the face-vertex $b_i$ has outdegree 2, %and the edge $\{b_i,v_i\}$ is oriented toward $v_i$,
\item the face-vertices of $\Gd$ distinct from $b_1,\ldots,b_5$ have outdegree 1, 
\item the original inner vertices of $G$ have outdegree 4.
\end{compactitem}
%% \end{compactitem}
%% \begin{compactitem}
%% \item for all $i\in [1:5]$, the face-vertex $b_i$ has outdegree 2 in $\cB$, and the edge $\{b_i,v_i\}$ is oriented toward $v_i$,
%% \item the face-vertices of $\Gd$ distinct from $b_1,\ldots,b_5$ have outdegree 1 in $\cB$, and the original inner vertices of $G$ have outdegree 4.
%% \end{compactitem}
The orientation $\cB$ of $\Gd$ is represented in Figure~\ref{fig:construction}(c).

It is easy to check (using the Euler relation), that in $\cB$ there is a unique oriented edge $e_*$ of $\Gd$ whose initial vertex is an outer vertex of $G$. Let $v_*$ be the initial vertex of $e_*$.
One can show that the orientation $\cB$ is \emph{accessible from $v_*$}, that is, for every vertex of $v$ of $\Gd$ there is a directed path from $v_*$ to $v$ in $\cB$. We leave the proof of this accessibility property to the reader, and only mention that it relies on the fact that $G$ has no cycle of length less than 5 containing a vertex in its interior. Since $\cB$ is accessible from $v_*$, there exists a spanning tree $T$ of $\Gd$ rooted at $v_*$ and such that every edge of $T$ is oriented from parent to children in~$\cB$. The spanning tree~$T$ is represented in Figure~\ref{fig:construction}(d).

Next we use $\cB$ and $T$ to define an orientation $\cO$ of the primal-dual completion $G^+$ of $G$. Let $S$ be the set of inner faces of $\Gd$. By definition, $S$ is in bijection with the inner edges of $G$, and the primal-dual completion $G^+$ of $G$ is obtained by 
\begin{compactitem}
\item adding an edge-vertex $x_s$ in each inner face $s\in S$ and joining it to the 4 incident vertices of $\Gd$,
\item adding an edge-vertex $x_i$ for all $i \in [1:5]$ and joining it to $v_i$, $v_{i+1}$ and $b_i$.
\end{compactitem}

Let $\ov T$ be the set of edges of $\Gd$ not in $T$. It is classical (for any spanning tree $T$ of a plane graph) that $\ov T$ is in bijection with the set $S$ of inner face of $\Gd$, where the bijection $\varphi$ associates to an edge $e\in \ov T$ the face $\varphi(e)\in S$ incident to $e$ and enclosed by the unique cycle in $T\cup \{e\}$. The bijection $\varphi$ is indicated in Figure~\ref{fig:construction}(d).
For $s\in S$ we denote by $t_s$ the terminal vertex of the edge $e\in \ov T$ such that $\varphi(e)=s$. We now orient $G^+$ by the following rule illustrated in Figure~\ref{fig:construction}(e):
\begin{compactitem}
\item for each face $s\in S$, the edge $(x_s,t_s)$ is oriented toward $t_s$, while the 3 other edges incident to $x_s$ are oriented toward $x_s$,
\item for each $i \in [1:5]$, the edge $(x_i,b_i)$ is oriented toward $b_i$.
\end{compactitem}

The resulting orientation $\cO$ of $G^+$ is indicated in Figure~\ref{fig:construction}(f). %It is easy to see that the orientation $\cO$ 
One can check that $\cO$ is a 5c-orientation of $G$: each edge-vertex has outdegree 1, each face-vertex has outdegree 2, and each inner original vertex has outdegree $5$ (we leave the proof to the reader). 
%\OB{I find the following explaination hard to follow. Also, the situation is not as described for outer vertices and boundary vertices. I think we should go without proof in this section, but if we do give a proof maybe we need to make it easier to follow.} \EF{Yes sure it is probably too cryptic as it is |} 
%\ef{(for edge-vertices it is clear, for face-vertices and original vertices it relies on the fact that $\varphi$ and the orientation rules inducea bijection between edges in $\ov T$ and edges of $G^+$ going out of edge-vertices, such that the ends of corresponding edges coincide; hence from $\Gd$ to $G^+\backslash\{x_1,\ldots,x_5\}$ the indegree of vertices decreases by 1 for non-root vertices and is unchanged for $v_*$, while the degree is unchanged for inner vertices and decreases by 1 for outer vertices).}
Moreover, the operations required to go from the regular orientation $\cA$ to the 5c-orientation $\cO$ can clearly be performed in linear time. This concludes the proof of Theorem~\ref{thm:main}.
}

\section{Graph drawing algorithm for 5c-triangulations}\label{sec:algo}
%\documentclass[letter]{amsart}
%\input{defs.tex}
%\author{Olivier Bernardi$^{*}$ \and \'{E}ric Fusy$^{\dagger}$ \and Shizhe Liang$^{+}$}
%\begin{document}

\subsection{Paths and regions}\label{sec:paths_regions}
%\documentclass{amsart}
%\input{defs.tex}
%\begin{document}
%\author{Olivier Bernardi$^{*}$ \and \'{E}ric Fusy$^{\dagger}$ \and Shizhe Liang$^{+}$}
%\subsection{Paths and regions}

In this section we explain how a 5c-wood gives rise to paths and regions associated to each vertex of a 5c-triangulation; see Figure~\ref{fig:paths-regions}(b).
%Let $G$ be a 5c-triangulation endowed with a 5c-wood $\cW = (W_1,...,W_5)$. 

Let $G$ be an undirected graph. A \emph{biorientation} of $G$ is an arbitrary subset of arcs of $G$ (so that each edge of $G$ is oriented in either 0, 1 or 2 directions). A biorientation is \emph{acyclic} if it contains no simple directed cycle, with the convention that two opposite arcs (coming from the same edge oriented in 2 directions) do not constitute a simple cycle. Suppose now that $G$ is a 5c-triangulation, and let $\cW=(W_1,\ldots,W_5)$ be a 5c-wood. For $i \in [1:5]$, we denote by $W_i^-$ the set of arcs obtained by reversing the arcs in $W_i$ (that is, taking the opposite arcs).

\begin{prop}\label{prop:acyclic}
Let $\cW=(W_1,\ldots,W_5)$ be a 5c-wood of $G$. For all $i\in [1:5]$, the biorientation \llncs{$\cO_i$ made of the union of arcs $W_i\cup W_{i-1}\cup W_{i+1}\cup W_{i-2}^-\cup W_{i+2}^-$}\addllncs{$\cO_i = W_i\cup W_{i-1}\cup W_{i+1}\cup W_{i-2}^-\cup W_{i+2}^-$} is acyclic. Consequently, for each pair $j,k\in [1:5]$, the biorientation %s $W_j\cup W_k$ and 
$W_j\cup W_k^-$ is acyclic.
\end{prop}

%The biorientation $\cO_3$ is represented in Figure~\ref{fig:paths-regions}. 
\addllncswa{The proof of Proposition \ref{prop:acyclic} can be found in Appendix~\ref{sec:pf_prop_acyclic}.} 
\addllncswoa{The proof of Proposition \ref{prop:acyclic} can be found in the appendix of~\cite{OB-EF-SL:5c-Schnyder}.}  
An example is shown in Figure~\ref{fig:paths-regions}(a). \llncs{It is the analogue of a well-known property of Schnyder woods~\cite{Schnyder:wood2,F01} (stating that for a Schnyder wood $\cW=(W_1,W_2,W_3)$, and for $i\in[1:3]$, the orientation $W_i\cup  W_{i-1}^-\cup W_{i+1}^-$ is acyclic).}

\fig{width=.9\linewidth}{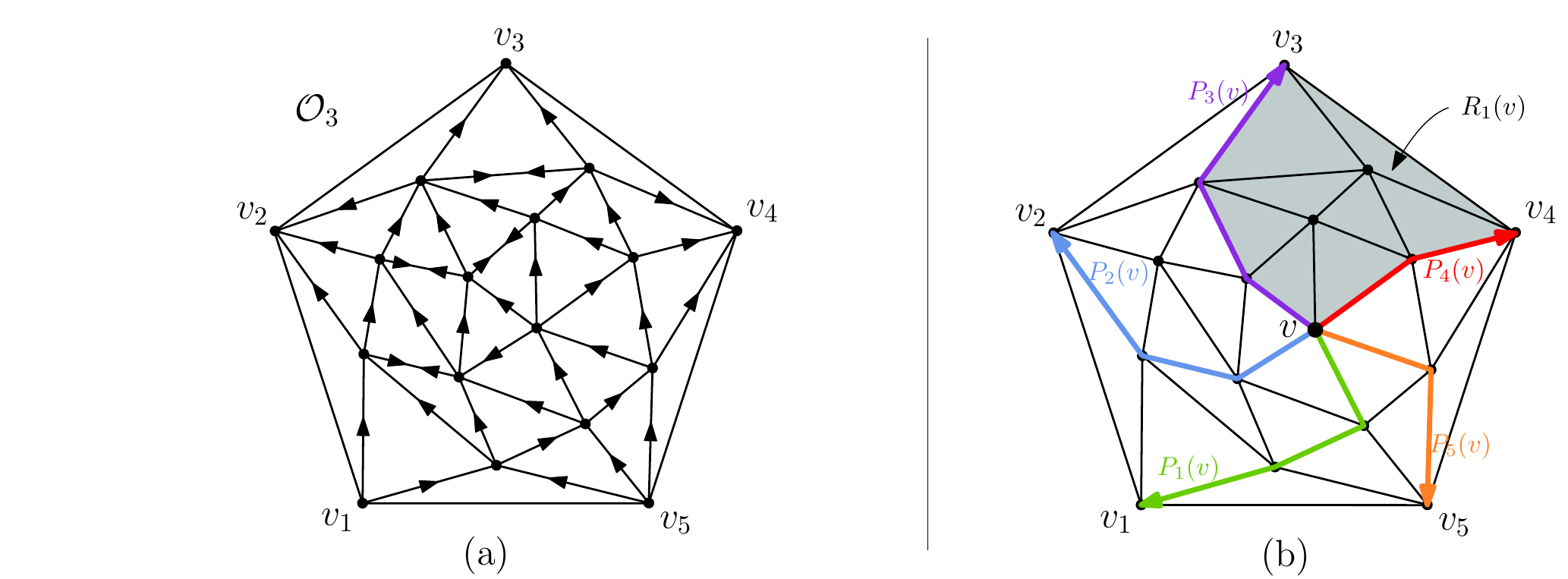}{(a) The biorientation $\cO_3$ for the 5c-wood represented in Figure~\ref{fig:5c_structures}. (b) The paths $P_1(v),\ldots, P_5(v)$ and the region $R_1(v)$ for the 5c-wood\llncs{represented} in Figure~\ref{fig:5c_structures}.}

\llncs{
\begin{proof}
 Suppose for contradiction that $\cO_i$ has a simple directed cycle. We put a partial order $\prec$ on the set of simple directed cycles of $\cO_i$, by declaring $C\prec C'$ if the region enclosed by $C$ is contained in the region enclosed by $C'$. Let $C$ be a directed cycle of $\cO_i$ which is minimal for this partial order.

Suppose first that there is a vertex $v_0$ in the region enclosed by $C$. For $j \in [1:5]$, let $P_j$ be the directed path starting at $v_0$ and made of the arcs in $W_j$. This directed path $P_j$ either ends at the outer vertex $v_j$ or contains a directed cycle, but in either case it does not stay within the region enclosed by $C$ (by minimality of $C$). Hence $P_{j}$ contains a directed path $P_j'$ from $v_0$ to a vertex of $C$. 
Hence $\cO_i$ contains a directed path (going through $v_0$) from a vertex of $C$ to a vertex of $C$: for instance, take the opposite of the path $P_{i-2}'$ concatenated with the path $P_i'$. This contradicts the minimality of $C$, hence \llncs{we conclude that}there is no vertex in the region enclosed by $C$. 

Thus, still by minimality of $C$, there cannot be any edge enclosed by $C$ (since every inner edge is oriented at least 1 way in $\cO_i$). Hence $C$ must be the contour of an inner face $f$ of $G$. If $C$ is a clockwise cycle, it is easy to see that the corners of $f$ are labeled either $i-1$ or $i-2$ in the 5c-labeling $\ov \Theta(\cW)$ (because the corner at a vertex $v$ of $f$ must be incident to the arc of $W_{i+1}$ with initial vertex $v$). But this is impossible since the 3 corners of~$f$ must have distinct labels in the 5c-labeling $\ov \Theta(\cW)$. Similarly, if $C$ is a counterclockwise cycle, then the corners of~$f$ are all labeled either $i+1$ or $i+2$ in $\ov \Theta(\cW)$, which is impossible. 

This completes the proof that $\cO_i$ is acyclic. The second statement follows immediately.
\end{proof}
}

As mentioned in Section~\ref{sec:combin}, Proposition~\ref{prop:acyclic} implies that in any 5c-wood $\cW=(W_1,\ldots,W_5)$ of $G$, each set $W_i$ is a tree directed toward the outer vertex $v_i$. For an inner vertex $v$ of $G$, we denote by $P_i(v)$ the directed path from $v$ to $v_i$ in $W_i$. These paths are indicated in Figure~\ref{fig:paths-regions}(b).
Proposition~\ref{prop:acyclic} implies that the paths $P_1(v),\ldots,P_5(v)$ have no vertex in common except $v$ (since if $P_j(v)$ and $P_k(v)$ had a vertex in common, the biorientation $W_j\cup W_k^-$ would contain a directed cycle). We denote by $R_i(v)$ the region enclosed by the simple cycle made of the outer edge $(v_{i+2},v_{i-2})$ together with the paths $P_{i-2}(v)$ and $P_{i+2}(v)$. See Figure~\ref{fig:paths-regions}(b) for an example.

%\addllncs{TODO: Reference paper \cite{nagai2000linear}.}\OB{Maybe it is enough to cite it in Question 4?}
\llncs{
\begin{remark}
For a graph $G$ with a designated marked vertex $r$, a tuple $(W_1,\ldots,W_k)$ of spanning trees of $G$ is called \emph{independent}~\cite{itai1988multi} if for any vertex $v\neq r$ in $G$, and any pair of indices $i\neq j$ in $[1:k]$, the paths from $v$ to $r$ in the trees $W_i$ and $W_j$ intersect only at $v$ and $r$. From the above property of 5c-woods, we conclude that any 5-connected triangulation $G$ with a marked vertex $r$ of degree $5$ admits an independent 5-tuple of spanning trees (constructed from a 5c-wood of $G'=G\backslash r$). It is not difficult to see that this also holds if the  marked vertex $r$ has degree larger than $5$, since $G'=G\backslash r$ can be completed into a 5c-triangulation by adding 5 vertices 
in the outer face forming an enclosing 5-gon, and adding edges to triangulate the annulus between the 5-gon and the former outer contour. This proves that every maximal 5-connected planar graph with a marked vertex admits an independent 5-tuple of spanning trees, which can be computed in linear time. 
This existence result was first established in~\cite{huck1999independent} (without the requirement on maximality), and the linear time property was proved in~\cite{nagai2000linear} by introducing a shelling order, called \emph{5-canonical decomposition}, specific to 5-connected triangulations. The 5-tuple of spanning trees computed in~\cite{nagai2000linear} partially fulfills the properties of 5c-woods,  but Conditions (W2) and (W3) do not hold in general (although some other crossing conditions are satisfied).
%may fail for one color, and Condition (W3) may also not hold.
\end{remark}
}

%% Let $G$ be a 5c-triangulation endowed with a 5c-wood $\cW = (W_1,...,W_5)$. 
%% Let $v$ be an inner vertex. For each $i \in [5]$, let $P_i(v)$ be the outgoing path of color $i$ starting from $v$ and ending at $v_{i-2}$. 
%% It follows from ?? that the five paths $P_1(v),...,P_5(v)$ are disjoint except for $v$ itself. 
%% %Note that Condition (QW0) implies that the endpoint of $P_i(v)$ is necessarily the outer vertex $v_{i-2}$.
%% %\begin{lemma}\label{lem:path-disjoint}
%% %Let $G$ be a 5c-triangulation endowed with a 5c-wood. Then, for every inner vertex $v$ of $G$, the five paths $P_1(v),...,P_5(v)$ are disjoint except for $v$ itself.
%% %\end{lemma}
%% Hence we can define regions $R_1(v),...,R_5(v)$ delimited by these paths, where $R_i(v)$ is the region between $P_{i-1}(v)$ and $P_i(v)$.

%\end{document}

\subsection{Graph drawing algorithm} 
For convenience, we extend the definition of $R_i(v)$ for $v$ an outer vertex by declaring $R_i(v_i)=G$ for all $i\in \set5$, and $R_i(v_j)=\{v_j\}$ for all $j\neq i$.
%if $v=v_{j}$ then $R_{i}(v)=G$, and $R_i(v)$ is reduced to the vertex $v$ for $i\neq i_0$. 
For a vertex $v$ of $G$ and for $i\in\set5$, the \emph{size} of $R_i(v)$, denoted by $|R_i(v)|$, is the number of inner faces in $R_i(v)$. Letting $n$ be the number of vertices of $G$, the Euler relation implies that $G$ has $2n-7$ inner faces. 
Let $\alpha_i(v):=|R_i(v)|/(2n-7)$. Since the inner faces of $G$ are partitioned among the 5 regions, we have $\sum_{i=1}^5\alpha_i(v)=1$. 

The \emph{5c-barycentric drawing algorithm} then consists of the following steps:
\begin{enumerate}
\item
Draw a regular pentagon $(v_1,v_2,v_3,v_4,v_5)$ (in clockwise order). %, with $(v_1,v_5)$ as horizontal bottom-edge.
\item
Place each vertex $v$ at the barycenter $\sum_{i=1}^5 \alpha_i(v) v_i$.
\item
Draw each edge of $G$ as a segment connecting the corresponding points.
\end{enumerate}

An output of this algorithm is represented in Figure~\ref{fig:5connbig_drawing}. In all our drawings,\llncs{ we have used the convention that} the outer pentagon is drawn with the edge $\{v_1,v_5\}$ as an horizontal bottom segment.

\fig{width=0.8\linewidth}{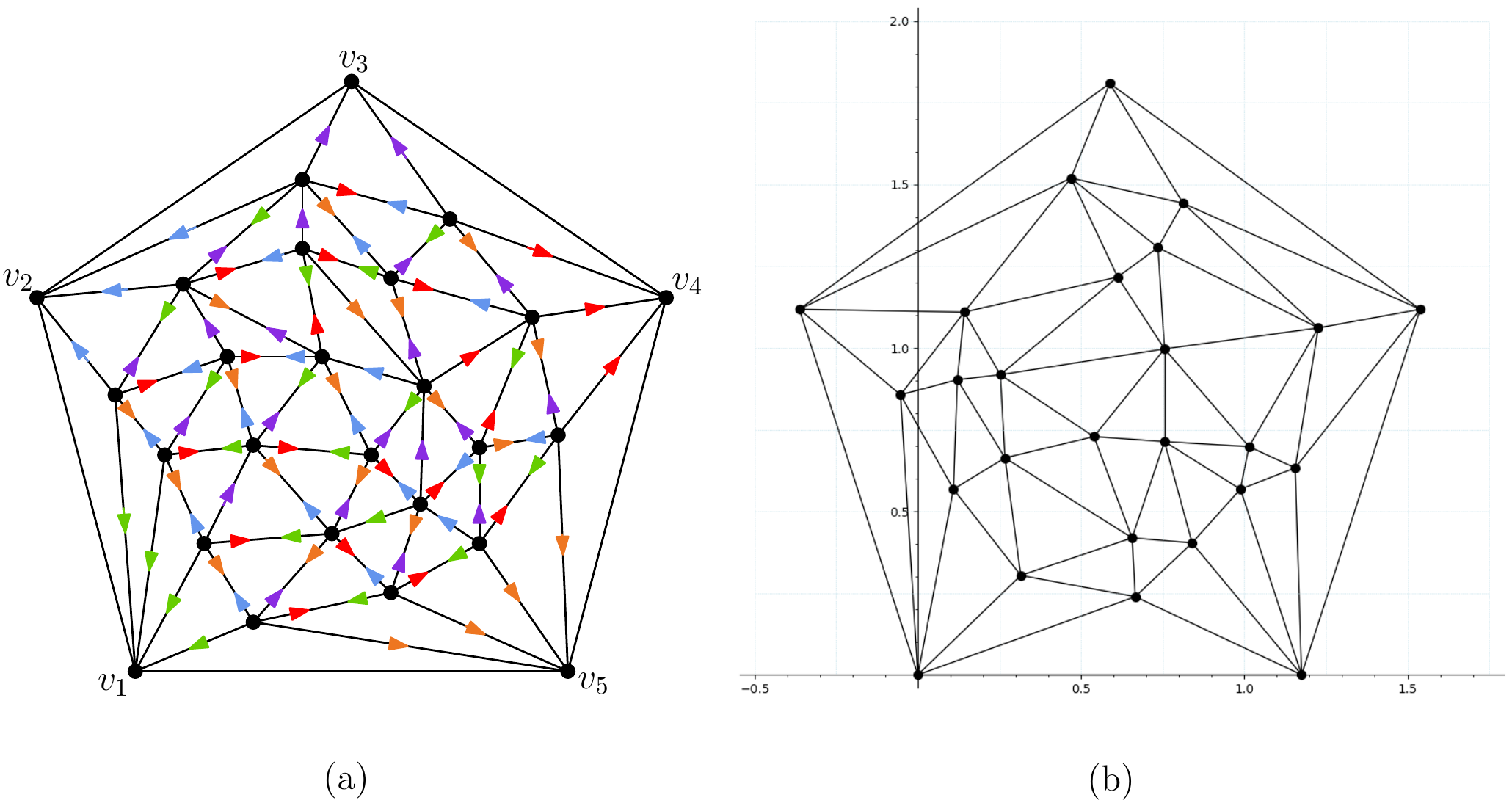}{(a) A 5c-wood $\cW=(W_1,\ldots,W_5)$ of a 5c-triangulation (the sets of arcs in $W_1,\ldots,W_5$ are represented by arrows of 5 different colors). (b) The drawing of $G$ obtained by applying the 5c-barycentric drawing algorithm with $\cW$ as input.}

\begin{theo}\label{theo:planar}
For each 5c-triangulation $G$ with $n$ vertices, the 5c-barycentric drawing algorithm yields a planar straight-line drawing of $G$, which can be computed in $O(n)$ operations.
\end{theo}

\llncs{The proofs of planarity and of the time complexity of the drawing are carried in the next two subsections.}

\begin{remark}
Compared to Schnyder's algorithm~\cite{Schnyder:wood2} and to Tutte's spring embedding~\cite{tutte1963draw}, 
it is important here that the 5 outer vertices are placed so as to form a regular pentagon, and we use this property in our proof of planarity.
% (note also that no general linear transform exists to transform an arbitrary convex pentagon into a regular one).
\end{remark}

\addllncs{We will prove the planarity of the 5c-barycentric drawing in the next section.}
\addllncswa{The time-complexity analysis can be found in Appendix \ref{appendix-time-complexity}.} 

\subsection{Proof of planarity}\label{sec:planarity}

\fig{width=\linewidth}{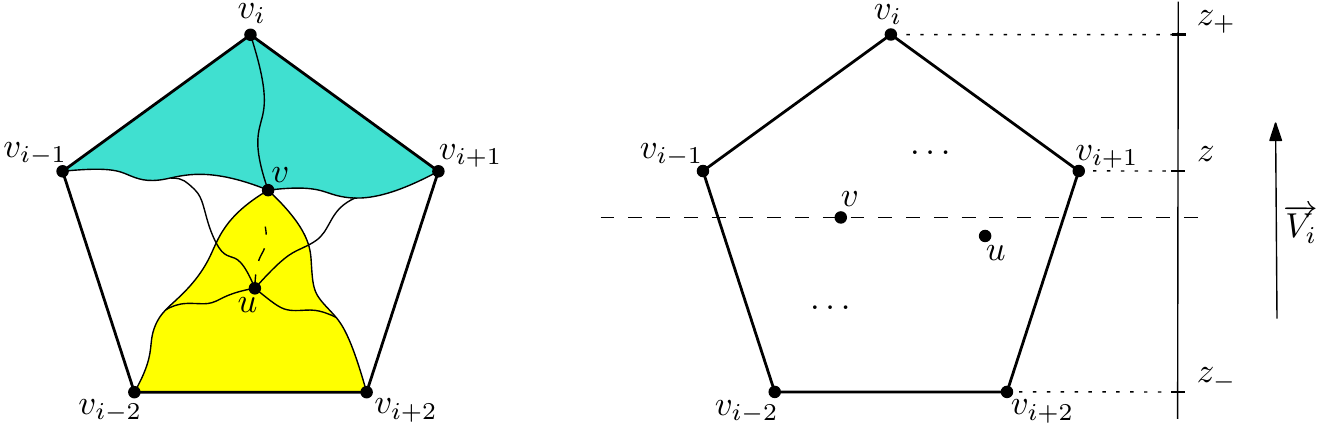}{Left: Situation for two vertices $u\neq v$ such that $u\in R_i(v)$, where the interior of $R_i(v)$ is shown in yellow, and 
the interior of $R_{i-2,i+2}(v)$ is shown in blue. Right: In the 5c-barycentric drawing, $u$ is in the half-plane ``below" $v$ upon rotating the drawing such that $\protect\ove{V_i}$ points upward.}

%\fig{width=8cm}{sectors}{Illustration of the property stated in Lemma~\ref{lem:sectors}.}

The proof of planarity in Theorem~\ref{theo:planar} crucially relies on a property (illustrated in the right drawing of Figure~\ref{fig:triangle_cw} and to be established in Lemma~\ref{lem:sectors}) stating that the directions of the 
arcs of each color from an inner vertex are constrained to be in certain cones. As a first step we show the following ``half-plane property", illustrated 
in Figure~\ref{fig:halfplane}.

\begin{lem}\label{lem:halfplane}
Let $G$ be a 5c-triangulation endowed with a 5c-wood. 
In the associated 5c-barycentric drawing, for $i\in\set5$, let $\ove{V_i}$ be the vector connecting the center of the outer pentagon (i.e. the center of the circle circumscribed to that pentagon) to the outer vertex $v_i$. 
Then, for any vertices $u\neq v$ of $G$ such that $u\in R_i(v)$, we have $\ove{V_i}\cdot \ove{vu}<0$, where $\ove{vu}$ is the vector connecting $v$ to $u$ in the 5c-barycentric drawing. 
\end{lem}
\begin{proof}
For each vertex $w$ and each subset $S$ of $\set5$, we use the extended notation $R_S(w)=\cup_{i\in S}R_i(w)$, and $\alpha_S(w)=\sum_{i\in S}\alpha_i(w)$. We then observe the following containment relations (see Figure~\ref{fig:halfplane}): %\begin{center}
\[
R_i(v)\supset R_i(u),\ \ \ R_{i-2,i+2}(v)\subset R_{i-2,i+2}(u). 
\]
The second relation is due to the fact (which follows from Condition~(W2) of 5c-woods) that, whenever $P_{i-1}(u)$ (resp. $P_{i+1}(u)$) 
leaves the region $R_i(v)$, it occurs just after visiting a vertex on $P_{i-2}(v)$ (resp. on $P_{i+2}(v)$). 

Defining $z_-<z<z_+$ as 
\[
z_-=\ove{V_i}\cdot\ove{V_{i-2}}=\ove{V_{i}}\cdot\ove{V_{i+2}},\ \ \ z=\ove{V_{i}}\cdot\ove{V_{i-1}}=\ove{V_{i}}\cdot\ove{V_{i+1}},\ \ \ z_+=\ove{V_{i}}\cdot\ove{V_{i}},
\] 
and letting $r=\alpha_{i-2,i+2}(u)-\alpha_{i-2,i+2}(v)$, $s=\alpha_{i-1,i+1}(u)-\alpha_{i-1,i+1}(v)$, and $t=\alpha_{i}(u)-\alpha_{i}(v)$, we then have 
%Assuming w.l.o.g. that the circumscribed circle has radius 1, we have 
\begin{align*}
\ove{vu}\cdot\ove{V_i}& = \sum_{k=1}^5(\alpha_k(u)-\alpha_k(v))\ove{V_k}\cdot\ove{V_{i}}\\
& = r\,z_-+s\,z+t\,z_+\\
& = r\,z_-+(-r-t)z+t\,z_+ = r\,(z_--z)+t\,(z_+-z),
\end{align*}
where from the 2nd to 3rd line we use the identity $r+s+t=1-1=0$. 
%where $s=\alpha_{i-2,i+2}(u)-\alpha_{i-2,i+2}(v)$ and $t=\alpha_{i}(u)-\alpha_{i}(v)$. 
The above containment relations ensure that $t<0$ and $r>0$, so that 
 $\ove{vu}\cdot\ove{V_i}<0$. \addllncs{\qed}
\end{proof}

\begin{lem}\label{lem:sectors}
Let $G$ be a 5c-triangulation endowed with a 5c-wood. For $v$ an inner vertex of $G$, and for $i\in\set5$, let $u$ be the terminal vertex of the arc of color $i$ from $v$. 
%let $a_i(v)$ be the arc of color $i$ having initial vertex $v$. 
Then, in the 5c-barycentric drawing of $G$, the angle between the vector $\ove{V_i}$ (defined as in Lemma~\ref{lem:halfplane}) and the vector $\ove{vu}$ 
%$a_i(v)$ (seen as a vector) 
is in the open interval $(-\frac{3\pi}{10},\frac{3\pi}{10})$ (see Figure~\ref{fig:triangle_cw} right).
\end{lem}
\begin{proof}
%Let $u$ be the end of $a_i(v)$. 
Clearly %$R_{i-2}(v) \cap R_{i+2}(v)=P_i(v)$. In particular, 
$u\in P_i(v)= R_{i-2}(v) \cap R_{i+2}(v)$. Hence, by Lemma~\ref{lem:halfplane} we have 
$\ove{vu}\cdot \ove{V_{i-2}}<0$ and $\ove{vu}\cdot \ove{V_{i+2}}<0$, which is equivalent to the stated property.  \addllncs{\qed}
\end{proof}

\begin{remark}
Lemmas~\ref{lem:halfplane} and~\ref{lem:sectors} are the analogues of well-known properties of Schnyder drawings for simple triangulations, where $(-\frac{3\pi}{10},\frac{3\pi}{10})$ is to be replaced by 
$(-\frac{\pi}{6},\frac{\pi}{6})$ in Lemma~\ref{lem:sectors}. A difference here is that the property as stated does not immediately guarantee that the outgoing edges of an inner vertex are in clockwise order 
in the drawing, since the cones for adjacent colors overlap. However, as we now show, the property is actually enough to ensure planarity, and thus the cyclic ordering of edges around a vertex is preserved in the drawing. 
\end{remark}

\fig{width=\linewidth}{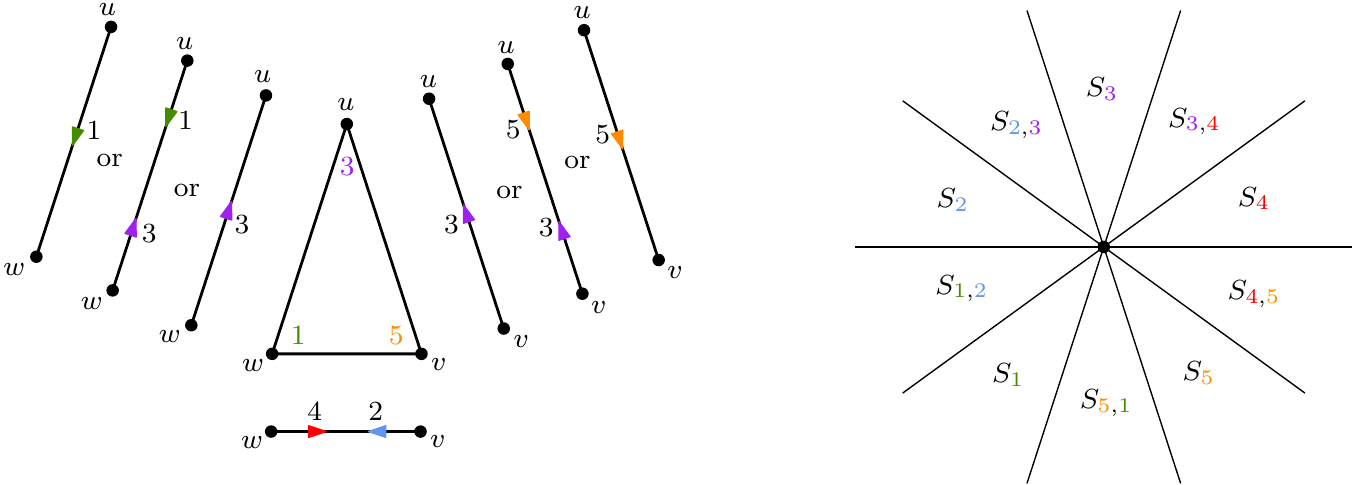}{\llncs{The left drawing shows the p}\addllncs{Left: P}ossible cases for the colors of arcs around an inner face with corner labels $1,3,5$. \llncs{The right drawing shows the c}\addllncs{Right: C}onditions 
of Lemma~\ref{lem:sectors} (a sector $S_i$ can only have the arc of color $i$ from the shown vertex, a sector $S_{i,j}$ can only have the arcs of colors in $\{i,j\}$).}

To show that the drawing is planar, by a known topological argument (see e.g.~\cite[Lem.4.4]{CDV:these}), 
it suffices to show that each inner face is \emph{properly represented}, that is, represented as a non-degenerated triangle that is not flipped (indeed, this condition ensures
that the function mapping a point $p$ inside the enclosing pentagon to the number $n(p)$ of triangles covering it is locally constant, hence has to be equal to $1$ everywhere). 
In other words, if the corner labels are $i,i+2,i+3$ in clockwise order around an inner face, we have to show 
that these corners are seen in the same cyclic order in the triangle representing the face. It is clear that the 5 inner faces incident to the 5 outer edges are properly represented, so we can focus on inner faces incident to 3 inner edges. 
Let us treat the case $i=3$ (the other cases can be treated symmetrically), where
the vertices at the corners $3,5,1$ are respectively denoted $u,v,w$. Given the conditions defining corner-labelings, and the bijection between 5c-woods and 5c-labelings, it is 
easy to check that the following holds (see Figure~\ref{fig:triangle_cw} left):
\begin{itemize}
\item
The arc $(v,w)$ has color $2$, and the arc $(w,v)$ has color $4$. 
\item
The arc $(v,u)$ has color $3$ if colored, the arc $(u,v)$ has color 5 if colored, and at least one of these two arcs is colored. 
\item
The arc $(w,u)$ has color $3$ if colored, the arc $(u,w)$ has color 1 if colored, and at least one of these two arcs is colored. 
\end{itemize}

By Lemma~\ref{lem:sectors}, and using the notation of Figure~\ref{fig:triangle_cw} right, the arc $(v,w)$ is in sector $S_2\cup S_{1,2}$ from $v$, and the arc $(w,v)$ is in sector $S_4\cup S_{4,5}$ from $w$ (the fact that $\{v,w\}$ is a straight segment in the drawing excludes sectors $S_{2,3}$ and $S_{3,4}$ respectively). 
If the arc $(v,u)$ is colored (with color $3$), then it is in sector $S_{2,3}\cup S_3\cup S_{3,4}$ from $v$, hence the directed angle from $(v,u)$ to $(v,w)$ 
around $v$ is in the open interval $(0,\pi)$, so that the face $(u,v,w)$ is properly represented. Similarly, if the arc $(w,u)$ is colored (with color $3$), then 
the directed angle from $(w,v)$ to $(w,u)$ around $w$ is in the open interval $(0,\pi)$, so that the face $(u,v,w)$ is properly represented. 
Finally, if none of the arcs $(v,u)$ or $(w,u)$ is colored, then the arc $(u,w)$ has color $1$, and the arc $(u,v)$ has color $5$. 
Assume for contradiction that the face $(u,v,w)$ is not properly represented. Then the directed angle from $(u,w)$ to $(u,v)$ around $u$ is not in the open interval $(0,\pi)$. Given Lemma~\ref{lem:sectors},
these two arcs have to be in $S_{5,1}$ from $u$, that is, the arc $(w,u)$ is in $S_3$ from $w$ and the arc $(v,u)$ is in $S_3$ from $v$. 
But then, the angle from $(w,v)$ to $(w,u)$ around $w$ is in the open interval $(\pi/5,4\pi/5)$, so that the face $(u,v,w)$ is properly represented. This concludes
the proof of planarity in Theorem~\ref{theo:planar}.

\llncs{\subsection{Time complexity} Similarly as for Schnyder's algorithm, the time-complexity of the 5c-barycentric drawing algorithm is linear in the number of vertices. Precisely, for a 5c-triangulation $G$ with $n$ vertices, we have already established in Theorem~\ref{thm:main} that a 5c-wood $\cW=(W_1,\ldots,W_5)$ can be computed in $O(n)$ operations. As we now explain, the regions' sizes $|R_i(v)|$ for all inner vertices $v$ and $i\in\set5$ can also be computed in $O(n)$ operations, hence the total time-complexity of the drawing algorithm is $O(n)$. 
\addllncs{The linear time computation is based on  following observation: for an inner vertex $v$, the vertices strictly inside the region $R_i(v)$ are all the descendants in the tree $W_i$ of the inner vertices on the paths $P_{i-2}(v)$ and $P_{i+2}(v)$ (excluding the vertices on the paths $P_{i-2}(v)$ and $P_{i+2}(v)$). Computing the number of such descendants is possible in linear time, similarly as in the case of Schnyder's algorithm \cite{Schnyder:wood2}.}

\llncs{The computation below is based on the following observation: for an inner vertex $v$, the vertices strictly inside the region $R_i(v)$ are all the descendants in the tree $W_i$ of the inner vertices on the paths $P_{i-2}(v)$ and $P_{i+2}(v)$ (excluding the vertices on the paths $P_{i-2}(v)$ and $P_{i+2}(v)$).
%$$for a 5c-triangulation $G$ with $n$ vertices, endowed with a 5c-wood $\cW=(W_1,\ldots,W_5)$, with $V$ the set of inner vertices of $G$, the time complexity of the drawing algorithm, i.e., the complexity of computing all sizes $|R_i(v)|$ over $v\in V$ and $i\in\set5$, is $O(n)$. 
Let $V$ be the set of inner vertices of $G$. For $v\in V$ and $i\in\set5$, let $N_i(v)$ be the number of descendants of $v$ in $W_i$ (with $v$ not considered a descendant of itself). 
The list $\{N_i(v),\ v\in V\}$ can be computed in $O(n)$ operations proceeding from leaves to root, that is, starting from the set $L$ of leaves of $W_i$, then the set of leaves of $W_i\backslash L$, etc. 
For $j\in\{i-2,i+2\}$, we let $N_i^j(v)$ be $N_i(w)+1$ if $v$ has a child $w$ in $W_i$ such that the arc $(v,w)$ has color $j$ (such a child is necessarily unique by Condition (W2) of 5c-woods), and be $0$ otherwise. Clearly the list $\{N_i^j(v),\ v\in V\}$ can be computed in $O(n)$ operations. 
Once all these lists are computed, for $i\in\set5$ and $v\in V$, and with the notation $P_k'(v)=P_k(v)\backslash\{v,v_k\}$, 
 we let $\Nil(v)=\sum_{u\in P_{i-2}'(v)}(N_{i}(u)-N_{i}^{i-2}(u))$, and $\Nir(v)=\sum_{u\in P_{i+2}'(v)}(N_{i}(u)-N_{i}^{i+2}(u))$.    
 The list $\{\Nil(v),\ v\in V\}$ (resp. $\{\Nir(v),\ v\in V\}$) can again be computed in $O(n)$ operations, here proceeding from root to leaves in $W_{i-2}$ (resp. in $W_{i+2}$). 
 Then, for $v\in V$, the above observation implies that the number $n_i(v)$ of vertices strictly inside $R_i(v)$ is 
 \[
 n_i(v)=\big(N_i(v)-N_i^{i-2}(v)-N_i^{i+2}(v)\big)+\Nil(v)+\Nir(v).
 \]
One can also compute the list $\{\mathrm{length}(P_i(v)),\ v\in V\}$ in $O(n)$ operations (proceeding from root to leaves in $W_i$). By the Euler relation the number of faces $|R_i(v)|$ is $2n_i(v)+\mathrm{length}(P_{i-2}(v))+\mathrm{length}(P_{i+2}(v))-1$. Hence the list $\{|R_i(v)|,\ v\in V\}$ can be computed in $O(n)$ operations, as claimed. }

\medskip
}

\subsection{Variations and other properties} We discuss here some variants (weighted faces, vertex-counting) and aesthetic properties of the drawing, and 
end with some open questions.

\medskip

\noindent\textbf{Variants of the embedding algorithm.}
 As in the case of Schnyder's algorithm~\cite{felsner2008schnyder}, one can give a \emph{weighted version} of the 5c-barycentric drawing algorithm. In this version, each inner face is assigned a positive weight, and $\alpha_i(v)$ is the total weight of inner faces in $R_i(v)$, divided by the total weight of inner faces. The proof of Lemma~\ref{lem:sectors} and the proof of planarity extend verbatim.   
Moreover, as in~\cite[Sec.7]{Schnyder:wood2}, a vertex-counting variant can also be given, upon changing $|R_i(v)|$ to be the number of vertices in $R_i(v)\backslash P_{i-2}(v)$,
and using $\alpha_i(v)=|R_i(v)|/(n-1)$, with $n$ the number of vertices of $G$. 
The relevant inequalities to prove Lemma~\ref{lem:halfplane} become, for $u\neq v$ such that $u\in R_i(v)$, %\addllncs{for $u \neq v, u\in R_i(v)$} 
\[
|R_i(v)|\geq |R_i(u)|,\ \ \ |R_{i-2,i+2}(v)|< |R_{i-2,i+2}(u)|.
\] 
(Note that, even if $R_i(v)\supset R_i(u)$, here we may have $|R_i(v)|= |R_i(u)|$ when $u,v$ are part of a triangular face $(v,u,w)$ such that $u$ is 
the end of the arc of color $i-2$ and $w$ is the end of the arc of color $i+2$ from $v$.) 
Hence, Lemma~\ref{lem:sectors} still holds, which as before implies planarity. 
% Let $G$ be a 5c-triangulation with $n$ vertices, endowed with a 5c-wood. 
%For $v$ an inner vertex of $G$, let $n_i(v)$ be the number of vertices in $R_i(v)\backslash P_{i+2}(v)$; for an outer vertex $v_{i_0}$ let $n_i(v)=(n-1)\delta_{i=i_0}$. 
%Then $\sum_{i=1}^5 n_i(v)=n-1$ for each vertex. We can then apply the barycentric 
%vertex placement rules using $\alpha_i(v)=n_i(v)/(n-1)$ for each vertex $v$ and $i\in \set5$.  Using the extended notation $n_S(v)=\sum_{i\in S}n_i(v)$ for a subset $S$ of $\set5$, 
%we can easily check that, for $a=(v,u)$ an arc of color $i$, the following inequalities hold:
%\[
%n_i(u)> n_i(v),\ \ \ n_{i,i-1}(u)>n_{i,i-1}(v),\ \ \ n_{i,i+1}(u)>n_{i,i+1}(v) 
%\]  
%This implies that Lemma~\ref{lem:sectors} still holds, replacing $]-3\pi/10,3\pi/10[$ by $]-3\pi/10,3\pi/10]$ in the statement. Then the planarity proof works the same as for the face-counting version.

\medskip

\noindent\textbf{Symmetries.}
Compared to Tutte's spring embedding~\cite{tutte1963draw} 
(and similarly as for Schnyder's drawing), the drawing algorithm (either in its face-counting or vertex-counting version) can display rotational symmetries, but not mirror symmetries. Precisely, there is a canonical 5c-wood corresponding to the \emph{minimal 5c-orientation}, that is, the unique 5c-orientation with no counterclockwise directed cycle. If $G$ is invariant by a rotation of order $5$, then so is the minimal 5c-orientation, and so is the corresponding 5c-wood under a shift of the arc colors (and indices of the outer vertices) by 1. Thus, the drawing obtained using this canonical 5c-wood displays the rotational symmetry. 

\fig{width=\linewidth}{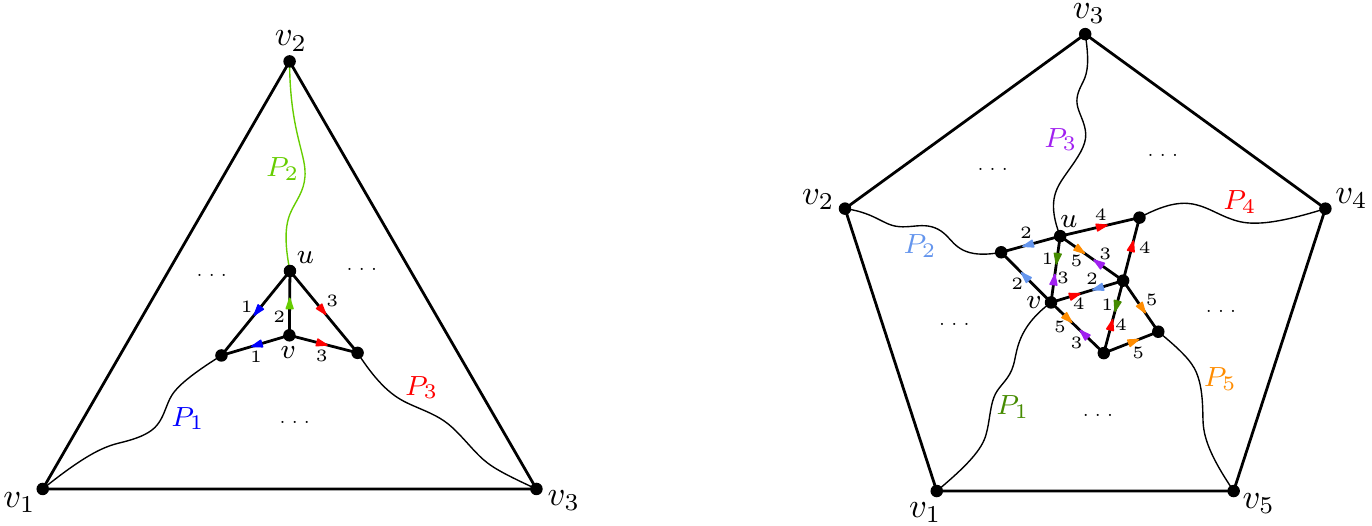}{The configuration (up to rotation and mirror) of vertex-pairs $(u,v)$ of smallest possible distance in Schnyder's drawing (left) and in our drawing (right).}

\medskip

\noindent\textbf{Quality of the drawing.}
Besides the fact that the family of triangulations is more restricted, an obvious disadvantage of our algorithm compared to Schnyder's one (and also to straight-line drawings using transversal structures~\cite{Fu07b}, or using shelling orders~\cite{FraysseixPP88,chrobak1997convex,he1997grid,MiuraNN01,miura2022grid}) 
is that we can not use an affine transformation to turn the drawing into one on a regular square-grid of linear width and height\footnote{If we use an affine transformation to have $v_1$, $v_2$ and $v_5$ placed at $(0,0)$, $(0,1)$ and $(1,0)$ respectively, then we get a drawing with vertex-coordinates in $\mathbb{Q}(\sqrt{5})$.}. 
Let us however discuss a parameter (vertex resolution) for which our algorithm brings some improvement.  % (assuming we compare them to draw 5c-triangulations), 
% (say, for a 5c-triangulation $G$, with $G'$ the triangulation obtained from $G$ by adding the edges $(v_2,v_4)$ and $(v_2,v_5$), we compare our drawing of $G$ to Schnyder's drawing of $G'$).   
For $n\geq 1$, we let $\mu_3(n)$ (resp. $\mu_5(n)$) 
denote the smallest possible distance between vertices over Schnyder's drawings on simple triangulations with $n$ vertices (resp. over our 5c-barycentric drawings on 5-connected triangulations with $n$ vertices),  assuming the drawing is normalized to have circumscribed circle of radius $1$. Below we use complex numbers to represent points in the plane.    
Let $\omt=e^{2i\pi/3}$ (resp. $\omf=e^{2i\pi/5}$). For a vector $\bdelta=(\delta_1,\delta_2,\delta_3)\in\mathbb{Z}^3$ (resp. $\bdelta=(\delta_1,\delta_2,\delta_3,\delta_4,\delta_5)\in\mathbb{Z}^5$) whose components add up to $0$, we define the \emph{modulus} $||\bdelta||$ of $\bdelta$ as the modulus $|z|$ 
of the complex number $z=\sum_{i=1}^3\delta_i\omt^i$ (resp. $z=\sum_{i=1}^5\delta_i\omf^i$).  
The vector $\bdelta$ is called \emph{possible} if there is a vertex pair $(u,v)$ in a Schnyder's drawing (resp. in our drawing) such that the number of faces in 
$R_i(u)$ minus the number of faces in $R_i(v)$ equals $\delta_i$ for $i\in\{1,2,3\}$ (resp. for $i \in [1:5]$), in which case the distance between $u$ and $v$ in the 
drawing is $||\bdelta||/(2n-5)$ (resp. is $||\bdelta||/(2n-9)$).  

As shown in Figure~\ref{fig:modulus} (and following an easy case inspection), the possible $\bdelta$ of smallest modulus (up to dihedral permutation of the entries) for Schnyder's drawing is 
 $(-1,2,-1)$, of modulus $d_3=3$, and for our drawing is $(-2,-1,3,1,-1)$, of modulus $d_5\approx 5.97$. %Having such a pattern is possible for any size $n\geq 4$ (resp. $n\geq $). 
Thus (for $n$ large enough) we have  $\mu_3(n)=\frac{d_3}{2n-5}$, and  $\mu_5(n)=\frac{d_5}{2n-9}$.  
Hence, the vertices are better kept away from each other in our drawing, the worst-case distance being increased by a factor $d_5/d_3\approx 1.99$ 
(to have a comparison over the same objects, we note that  $\mu_3(n)$ is also attained by drawings of 5-connected triangulations with $n$ vertices, for $n$ large enough).

In the vertex-counting variant, the possible $\bdelta$ of smallest modulus (up to cyclic permutation of the entries) in Schnyder's algorithm is now $(0,1,-1)$, of modulus $d_3'=\sqrt{3}\approx 1.73$, 
and in our drawing is $(-1,-1,1,1,0)$, of modulus $d_5'\approx 3.08$ (these are again attained by the patterns shown in Figure~\ref{fig:modulus}).   
Then we have (for $n$ large enough) $\mu_3(n)=\frac{d_3'}{n-1}$ 
 and $\mu_5(n)=\frac{d_5'}{n-2}$, so that the worst-case distance is increased by a factor $d_5'/d_3'\approx 1.78$ (in both algorithms the vertex-counting variant  brings a slight improvement over the face-counting version).  
 
 \medskip

\noindent \textbf{Open questions:} 
We conclude with some open questions: 
\begin{itemize}
\item[Question 1.] Can 5c-woods, and the 5c-barycentric algorithm, be generalized to 5-connected plane graphs in the spirit of the extension of Schnyder woods, and\llncs{ of the corresponding} drawing algorithm, to 3-connected planar graphs~\cite{bonichon2007convex,di1999output,F01,Miller:FelsnerWoods}?
% and of the corresponding drawing algorithms~\cite{F01,bonichon2007convex}?} 
\item[Question 2.] Can 5c-woods be used to define a graph drawing algorithm for the dual of 5c-triangulations (possibly with bent edges but restrictions on the directions of the edge segments). Such algorithms are known for the dual of triangulations~\cite{felsner2003geodesic,Miller:FelsnerWoods}, the dual of \addllncs{irreducible} triangulations of the 4-gon\llncs{ with no vertex inside triangles}~\cite{He93:reg-edge-labeling}, and the dual of quadrangulations~\cite{OB-EF:Schnyder,Biedl98abetter,Ta87}.
\item[Question 3.] \addllncs{Is there a shelling procedure on 5c-triangulations to output a 5c-wood (perhaps by adapting the 5-canonical decomposition introduced in~\cite{nagai2000linear})?} %(which outputs a 5-tuple of independent spanning trees satisfying certain crossing conditions).}
\llncs{Can a shelling procedure be used on 5c-triangulations (perhaps modifying the 5-canonical decomposition introduced in~\cite{nagai2000linear}) to output a 5c-wood?} 
\item[Question 4.] \addllncs{Can 5c-orientations be used (e.g. using the framework of~\cite{OB-EF:girth}) to have a bijective derivation of the generating function of 5-connected 
 triangulations, expressed in~\cite{gao2001counting}?}
\llncs{An algebraic expression for the generating function of 5-connected triangulations was obtained in~\cite{gao2001counting} using a substitution approach.  Can 5c-orientations be used to have a bijective derivation of this expression? As a first step, we are optimistic that the framework of~\cite{OB-EF:girth} can be used to obtain a bijection between 5c-triangulations (via their 5c-orientation with no counterclockwise cycle) and a class of decorated trees\addllncs{.}\llncs{ defined by some local degree conditions.}} 
\item[Question 5.] Is there a nice counting formula for the total number of 5c-woods on 5c-triangulations with $n$ inner vertices? (For Schnyder woods such a formula is $\frac{6(2n)!(2n+2)!}{n!(n+1)!(n+2)!(n+3)!}$~\cite{Bonichon:realizers}.)
\end{itemize}
%\end{document}

%\bigskip

%% \noindent{\bf Acknowledgments.} Olivier Bernardi was partially supported by NSF Grant DMS-2154242. \'Eric Fusy was partially supported by the project ANR19-CE48-011-01 (COMBIN\'E), and the project ANR-20-CE48-0018 (3DMaps). 

\bibliographystyle{plain}
\bibliography{biblio-Schnyder}

\addllncswa{
\appendix
\section{Appendix}
In this appendix we prove Lemmas \ref{lem:bij_label_orient}, \ref{lem:bij_label_wood}, and Proposition~\ref{prop:acyclic}. We also complete the proof of Theorem \ref{thm:main} by describing a linear time algorithm for constructing 5c-structures. Lastly, we prove that the complexity of the drawing algorithm is linear in the number of vertices.

\subsection{Proof of Lemma \ref{lem:bij_label_orient}}\label{sec:pf_lem_bij_label_orien}

% Statement & proof of Lemma 2.3 if we agree to move it here.
We start by proving an additional property of 5c-labelings. 

% Property: sum of label jumps
\begin{lem}\label{lem:sum_jumps}
 Let $\cL$ be a 5c-labeling of a 5-triangulation $G$. For any inner edge $e$, the sum of label jumps between consecutive corners in counterclockwise order around $e$ is equal to 5. This is represented in Figure~\ref{fig:jumps_around_edge}.
%consider the four corners incident to $e$ and the four label jumps in counterclockwise order around $e$ as represented in Figure~\ref{fig:jumps_around_edge}. The sum of labels jumps counterclockwise around $e$ is equal to 5.
\end{lem}

\begin{proof}
\llncs{As noted above,}\addllncs{Note that} for any inner vertex or any inner face $x$, the sum $\cwjump(x)$ of label jumps between consecutive corners in clockwise order around $x$ is equal to 5. For an inner edge~$e$ we denote by $\ccwjump(e)$ the sum of label jumps between consecutive corners in counterclockwise order around~$e$. 

% Let $\cwjump(v)$ (resp. $\cwjump(f)$) denote the sum of label jumps between consecutive corners in clockwise order around an inner vertex $v$ (resp. inner face $f$), and let $\ccwjump(e)$ denote the sum of label jumps in counterclockwise order around an inner edge. By conditions (L1) and (L2), $\cwjump(v) =5$ for every inner vertex $v$ and $\cwjump(f) = 5$ for every inner face $f$.

Observe that a clockwise label jump around an inner vertex or an inner face is either a counterclockwise label jump around an inner edge, or a label jump along an outer edge (and there are exactly 5 such jumps, each with value $1$). Hence, 
$\ds \sum_{f \in F} \cwjump(f) + \sum_{v \in V} \cwjump(v) = 5 + \sum_{e \in E} \ccwjump(e),$
%\begin{equation*}
%  \sum_{f \in F} \cwjump(f) + \sum_{v \in V} \cwjump(v) = 5 + \sum_{e \in E} \ccwjump(e),
%\end{equation*} 
where $V, E, F$ are the sets of inner vertices, inner faces and inner edges, respectively. The left-hand side is equal to $5(|V|+|F|)$, which in turn is equal to $5+5|E|$ by the Euler relation. Since $\ccwjump(e)$ must be a nonzero (by Condition (L2)) multiple of 5, it has to be exactly 5, for all $e \in E$.\qed
\end{proof}

\fig{width=0.3\linewidth}{jumps_around_edge}{Counterclockwise label jumps around an edge.}

%\begin{proof}[Proof of Lemma \ref{lem:bij_label_orient}]
Now we prove Lemma \ref{lem:bij_label_orient}. Let $\cL\in\bL_G$. 
We first show that $\Phi(\cL)$ is a 5c-orientation. Condition (O0) is a direct consequence of Condition (L0). Every inner primal vertex has outdegree 5 by Condition (L1), and every dual vertex has outdegree 2 by Condition (L2). 
Also, Lemma~\ref{lem:sum_jumps} implies that every edge-vertex has outdegree 1. Hence, Condition (O1) holds, and  $\Phi(\cL)$ is a 5c-orientation. 

%We prove $\Phi$ is a bijection by providing its inverse map. 
We now describe the inverse map. 
Let $\mathcal{O}$ be a 5c-orientation. The orientations of edges in $\cO$ dictate the label jumps between consecutive corners around vertices and faces. 
One can then follow these label jumps in order to propagate labels, starting from corners at the outer vertices (which are fixed by Condition (L0)) inward to all corners. If no conflict arises in this propagation (so that all the dictated label jumps are satisfied), then one obtains a labeling of corners $\cL$, which we denote by $\ov{\Phi}(\cO)$. We prove below that, for any 5c-orientation $\cO$, no conflict occurs during the propagation of labels, and that $\ov{\Phi}(\cO)$ is a 5c-labeling.

% let $\cL$ be the resulting labeling of corners, and let $\ov{\Phi}$ be the mapping associating $\cL$ to $\cO$. 
%To show that $\cL$ (and thus $\ov{\Phi}$) is well-defined, we need to prove that no conflicts occurs during the propagation of labels.

Consider the \emph{corner graph} $C_G$ of $G$ defined as follows: $C_G$ is a directed graph whose vertices are the inner corners of $G$, and there is an oriented edge from a corner $c$ to a corner $c'$ in $C_G$ if $c'$ is the corner following $c$ in clockwise order around a face or a vertex. A corner graph is represented in Figure~\ref{fig:corner_graph}. $C_G$ is naturally endowed with an embedding which is induced by the embedding of $G$ (in fact $C_G$ is obtained from the dual of $G^+$ by deleting a vertex). In particular, $C_G$ has three types of inner faces, corresponding to inner vertices, inner edges and inner faces of $G$ respectively.

We define the $\cO$\emph{-weight} $w(a)$ of an arc $a = (c,c')$ of $C_G$ to be the label jump from $c$ to $c'$ determined by $\mathcal{O}$ according to the rules given by Figure~\ref{fig:local_rules_5c_llncs}(a). The $\cO$\emph{-weight} $w(P)$ of a directed path $P$ of $C_G$ is the sum of the $\cO$-weight of the arcs of $P$. Note that the propagation rule causes no conflicts if and only if for any two (not necessarily consecutive) corners $c$ and $c'$ and any two directed paths $P_1, P_2$ in $C_G$ from $c$ to $c'$, we have $w(P_1) \equiv w(P_2) \pmod 5$. To show the later property it suffices to show that for any simple cycle $C$ in $C_G$, we have 
$$\sum_{a \in C^+}w(a) - \sum_{a \in C^-}w(a) \equiv 0 \pmod 5,$$ 
where $C^+$ and $C^-$ are the sets of arcs appearing clockwise and counterclockwise on $C$, respectively. It is easy to see that this holds if and only if the $\cO$-weight of the contour of each inner face in $C_G$ is congruent to 0 modulo 5. This last condition is implied by Condition (O1) of 5c-orientations.  
%Condition (O1) of 5c-orientations ensures that the $\cO$-weight of the contour of each inner face in $C_G$ is congruent to 0 modulo 5 \SL{(what about: but this is ensured by Condition (O1) of 5c-orientations?)}. 
Hence the corner labeling $\cL=\ov{\Phi}(\cO)$ is well-defined. Moreover $\cL$ is indeed a 5c-labeling because Condition (L0) is satisfied by definition, and Conditions (L1) and (L2) are easy consequences of Condition (O1). 

%% Then Condition (O1) implies that the weight of the contour of each type of inner face in $C_G$ is equal to 5.

%% Note that the propagation rule causes no conflicts if and only if for any two (not necessarily consecutive) corners $c$ and $c'$ and any two directed paths $P_1, P_2$ in $C_G$ from $c$ to $c'$, we have $w(P_1) \equiv w(P_2) \pmod 5$. To show the later condition it suffices to show that for any simple cycle $C$ in $C_G$, we have $$\sum_{a \in C^+}w(a) - \sum_{a \in C^-}w(a) \equiv 0 \pmod 5,$$ where $C^+$ and $C^-$ are the sets of directed edges appearing clockwise and counterclockwise on $C$, respectively. It is easy to see that this holds if and only if the weight of the contour of each inner face in $C_G$ is congruent to 0 modulo 5. Hence $\ov{\Phi}$ is well-defined. Also, the resulting corner labeling $\cL$ is indeed a 5c-labeling because Condition (L0) is satisfied by definition of $\ov{\Phi}$ and (L1) and (L2) are easy consequences of Condition (O1). 

Finally, it is easy to see that $\Phi$ and $\ov{\Phi}$ are inverse to each other, hence they give bijections between 5c-labelings and 5c-orientations of $G$.
This completes the proof of Lemma \ref{lem:bij_label_orient}.
%\end{proof}

\fig{width=0.8\linewidth}{corner_graph}{A 5-triangulation $G$ and its corner graph $C_G$.}

\subsection{Proof of Lemma \ref{lem:bij_label_wood}}\label{sec:lem_bij_label_wood}
%\begin{proof}[Proof of Lemma \ref{lem:bij_label_wood}]
Let $\cL$ be a 5c-labeling of $G$.
 We first show that $\Theta(\cL)$ is a 5c-wood. By (L0), the arcs that start at outer vertices receive no color. 
 Moreover, Conditions~(L1) and~(L2) and Lemma~\ref{lem:sum_jumps} imply that for any inner edge $e=\{u,v\}$ with labels $i$ on each side at $v$, the labels at $u$ must be $i+2$ 
 on the left and $i+3$ on the right, so that the arc $(u,v)$ receives label $i$. 
 This implies that the inner arcs ending at $v_i$ receive color $i$ by $\Theta$. Hence $\Theta(\cL)$ satisfies (W0).

 It is clear that (W1) is a direct consequence of (L1). To verify (W2), consider an inner arc $a=(u,v)$ of color $i\in\set5$, with $v$ an inner vertex.   
 Let $i+2, i+3, j, k$ be the labels in counterclockwise order around $a$, starting with those two incident to $u$.  %as indicated in Figure~\ref{fig:outer_labels}(b). 
 By (L2), $j \in \{i-1,i\}$ and $k \in \{i,i+1\}$. Therefore the opposite arc $-a$ has either no color while being 
 between the arcs of colors $i+2$ and $i+3$ starting from $v$, or it has color $i+2$, or color $i+3$, which proves that the arc $a$ satisfies (W2).

 Lastly, for any inner edge $e=\{u,v\}$,  
 %recall that Lemma~\ref{lem:sum_jumps} guarantees that the sum of labels jumps in counterclockwise order around an inner edge $e$ is always 5. 
 Lemma~\ref{lem:sum_jumps} and Conditions~(L1) and (L2) easily imply that either 1 or 2 arcs of $e$ receive a color (depending on having one or two  
  label jumps equal to $2$ along $e$, giving respectively two or one label jumps equal to 1 across $e$).   
 Therefore $\Theta(\cL)$ satisfies (W3) and is a 5c-wood.

 To prove that $\Theta$ is a bijection, we now describe the inverse map. Given a 5c-wood $\cW$ on $G$, we define a corner labeling $\ov{\Theta}(\cW)$ as follows:
 \begin{itemize}
\item  the inner corners incident to the outer vertex $v_i$ receive label $i$; 
\item a corner incident to an inner vertex $v$ receives label $i$ if it is between the outgoing arcs of colors $i+2$ and $i+3$ in clockwise order around $v$.
\end{itemize}

We need to show that $\ov{\Theta}(\cW)$ is indeed a 5c-labeling. Conditions (L0) and (L1) are clearly satisfied. 
Next we check that any label jump in clockwise direction around a face is equal to~1 or~2. Let $e=\{u,v\}$ be an inner edge of $G$. 
By (W3), one of its two arcs has some color $i$, say the one starting at $u$. Let $i+2, i+3, j, k$ be the labels of the corners in counterclockwise order around $e$, starting
with those two incident to $u$.  %as indicated in Figure~\ref{fig:outer_labels}(b). 
Then (W2) implies that $(j,k) \in \{(i-1,i), (i,i), (i,i+1)\}$. In all cases, the label jumps from $i+3$ to $j$ and from $k$ to $i+2$ are both in $\{1,2\}$. Since the sum of label jumps in clockwise direction around a face is a multiple of 5, and each jump is in $\{1,2\}$, the only possibility is that one label jump is equal to 1 and two label jumps are equal to 2. This shows that $\ov{\Theta}(\cW)$ satisfies (L2), hence is a 5c-wood. 

 Finally, it is clear that the mappings $\Theta$ and $\ov{\Theta}$ are inverse of each other, hence they give bijections between the sets of 5c-labelings and 5c-woods of $G$. This completes the proof of Lemma \ref{lem:bij_label_wood}.
%\end{proof}

\subsection{Proof of Theorem \ref{thm:main} and construction algorithm}\label{sec:pf_main}
In this subsection we complete the proof of Theorem \ref{thm:main}.
Let us first show that a 5-triangulation $G$ admitting a 5c-orientation is necessarily a 5c-triangulation. If $G$ admits a 5c-orientation, it also admits a 5c-wood $\cW=(W_1,\ldots,W_5)$. For an inner vertex $v$ of $G$, we consider the paths $P_i(v)$ in $W_i$ from $v$ to $v_i$, for $i \in [1:5]$. 
As we have seen in Section~\ref{sec:paths_regions} (up to the forthcoming proof of Proposition~\ref{prop:acyclic}), the paths $P_1(v),\ldots,P_5(v)$ have no vertex in common except~$v$.  Hence $v$ cannot be in the interior of a simple cycle of length less than 5. This proves that $G$ is a 5c-triangulation.

In the rest of this section, we describe a linear time algorithm for computing a 5c-orientation of a 5c-triangulation $G$. The process is illustrated in Figure~\ref{fig:construction}, it uses  some facts about \emph{regular edge labelings}, which are combinatorial structures for triangulations of the 4-gon first defined by He in \cite{He93:reg-edge-labeling}, and rediscovered in \cite{Fu07b} under the name of \emph{transversal structures}. We will rely on an incarnation of these structures as certain outdegree-constrained orientations.  
Let $H$ be a triangulation of the 4-gon, and let $\Hd$ be the map obtained by placing a \emph{face-vertex} in each inner face of $H$ and joining it to the three vertices incident to this face (and then erasing the original edges of $H$). 

A \emph{regular orientation} of $H$ is an orientation of $\Hd$ such that every face-vertex of $\Hd$ has outdegree 1 and every original inner vertex of $H$ has outdegree 4; an example is shown in Figure~\ref{fig:construction}(b). It follows from \cite{He93:reg-edge-labeling,Fu07b} that a triangulation of the 4-gon $H$ admits a regular orientation if and only if $H$ has no cycle of length less than 4 containing a vertex in its interior (equivalently, it has no loop nor multiple edges, and every 3-cycle is the boundary of a face). Moreover, such an orientation can be computed in linear time~\cite{KantHe97:reg-edge-labeling-linear} (an independent proof is given in \cite{OB-EF-SL:Grand-Schnyder}).

%We call \emph{regular edge labeling} of $H$ an orientation of $\Hd$ such that every face-vertex of $\Hd$ has outdegree 1 and every original inner vertex of $H$ has outdegree 4; an example is shown in Figure~\ref{fig:construction}(b). It is shown in \cite{He93:reg-edge-labeling} that a triangulation of the 4-gon $H$ admits a regular edge labeling if and only if every 3-cycle of $H$ is the boundary of a face. Moreover, it is shown in \cite{KantHe97:reg-edge-labeling-linear} that such a structure can be computed in linear time in the number of vertices (an independent proof is given in \cite{OB-EF-SL:Grand-Schnyder}). 

\fig{width=\linewidth}{construction}{Construction of a 5c-orientation for a 5c-triangulation. (a) A 5c-triangulation $G$. (b) The triangulation of the $4$-gon $H$ (obtained from $G$ by placing a vertex $v_0$ in the outer face of $G$), and the regular orientation $\cA$ of $\Hd$. (c) The orientation $\cB$ of $\Gd$. (d) The spanning tree $T$ (indicated by bold lines) oriented from the root vertex $v_*$ to the leaves, and the bijection $\varphi$ between the set $S$ of inner faces of $\Gd$ and set $\ov T$ of edges not in $T$: for each edge $e$ in $\ov T$, the corresponding face $\varphi(e)$ is indicated by a shaded triangle inside of $\varphi(e)$ and incident to $e$. (e) The rule (inside each inner face $s$ of $\Gd$) for computing the orientation $\cO$ of the primal-dual completion $G^+$. (f) The resulting 5c-orientation $\cO$ of $G^+$.}

Let $G$ be a 5c-triangulation. Let $H$ be the triangulation of the 4-gon obtained from $G$ by placing a new vertex $v_0$ in the outer face of $G$ and joining $v_0$ to $v_1$, $v_2$, $v_3$ and $v_4$. 
Since $G$ is a 5c-triangulation, $H$ has no cycle of length less than 4 containing a vertex in its interior. Hence, as recalled above, $H$ admits a regular orientation $\cA$, which can be computed in linear time. The situation is represented in Figure~\ref{fig:construction}(a-b).
Let $\Gd$ be the subgraph of $\Hd$ obtained by deleting the vertices and edges of $\Hd$ in the outer face of $G$. 
Consider the restriction $\cA'$ of $\cA$ to $\Gd$. Let $b_i$ be the face-vertex of $\Gd$ in the inner face of $G$ incident to the outer edge $\{v_i,v_{i+1}\}$ for $i\in [1:5]$. By definition, the face-vertex $b_i$ has outdegree 1 in $\cA'$, and we can reorient one of the edges $\{b_i,v_i\}$ or $\{b_i,v_{i+1}\}$ to obtain an orientation $\cB$ of $\Gd$ such that 
\begin{compactitem}
\item for all $i\in [1:5]$, the face-vertex $b_i$ has outdegree 2, %and the edge $\{b_i,v_i\}$ is oriented toward $v_i$,
\item the face-vertices of $\Gd$ distinct from $b_1,\ldots,b_5$ have outdegree 1, 
\item the original inner vertices of $G$ have outdegree 4.
\end{compactitem}
%% \end{compactitem}
%% \begin{compactitem}
%% \item for all $i\in [1:5]$, the face-vertex $b_i$ has outdegree 2 in $\cB$, and the edge $\{b_i,v_i\}$ is oriented toward $v_i$,
%% \item the face-vertices of $\Gd$ distinct from $b_1,\ldots,b_5$ have outdegree 1 in $\cB$, and the original inner vertices of $G$ have outdegree 4.
%% \end{compactitem}
The orientation $\cB$ of $\Gd$ is represented in Figure~\ref{fig:construction}(c).

It is easy to check (using the Euler relation), that in $\cB$ there is a unique oriented edge $e_*$ of $\Gd$ whose initial vertex is an outer vertex of $G$. Let $v_*$ be the initial vertex of $e_*$.
One can show that the orientation $\cB$ is \emph{accessible from $v_*$}, that is, for every vertex of $v$ of $\Gd$ there is a directed path from $v_*$ to $v$ in $\cB$. We leave the proof of this accessibility property to the reader, and only mention that it relies on the fact that $G$ has no cycle of length less than 5 containing a vertex in its interior. Since $\cB$ is accessible from $v_*$, there exists a spanning tree $T$ of $\Gd$ rooted at $v_*$ and such that every edge of $T$ is oriented from parent to children in~$\cB$. The spanning tree~$T$ is represented in Figure~\ref{fig:construction}(d).

Next we use $\cB$ and $T$ to define an orientation $\cO$ of the primal-dual completion $G^+$ of $G$. Let $S$ be the set of inner faces of $\Gd$. By definition, $S$ is in bijection with the inner edges of $G$, and the primal-dual completion $G^+$ of $G$ is obtained by: 
\begin{compactitem}
\item adding an edge-vertex $x_s$ in each inner face $s\in S$ and joining it to the 4 incident vertices of $\Gd$,
\item adding an edge-vertex $x_i$ for all $i \in [1:5]$ and joining it to $v_i$, $v_{i+1}$ and $b_i$.
\end{compactitem}

Let $\ov T$ be the set of edges of $\Gd$ not in $T$. It is classical (for any spanning tree $T$ of a plane graph) that $\ov T$ is in bijection with the set $S$ of inner face of $\Gd$, where the bijection $\varphi$ associates to an edge $e\in \ov T$ the face $\varphi(e)\in S$ incident to $e$ and enclosed by the unique cycle in $T\cup \{e\}$. The bijection $\varphi$ is indicated in Figure~\ref{fig:construction}(d).
For $s\in S$ we denote by $t_s$ the terminal vertex of the edge $e\in \ov T$ such that $\varphi(e)=s$. We now orient $G^+$ by the following rule illustrated in Figure~\ref{fig:construction}(e):
\begin{compactitem}
\item for each face $s\in S$, the edge $(x_s,t_s)$ is oriented toward $t_s$, while the 3 other edges incident to $x_s$ are oriented toward $x_s$,
\item for each $i \in [1:5]$, the edge $(x_i,b_i)$ is oriented toward $b_i$.
\end{compactitem}

The resulting orientation $\cO$ of $G^+$ is indicated in Figure~\ref{fig:construction}(f). %It is easy to see that the orientation $\cO$ 
One can check that $\cO$ is a 5c-orientation of $G$: each edge-vertex has outdegree 1, each face-vertex has outdegree 2, and each inner original vertex has outdegree $5$ (we leave the proof to the reader). 
%\OB{I find the following explaination hard to follow. Also, the situation is not as described for outer vertices and boundary vertices. I think we should go without proof in this section, but if we do give a proof maybe we need to make it easier to follow.} \EF{Yes sure it is probably too cryptic as it is |} 
%\ef{(for edge-vertices it is clear, for face-vertices and original vertices it relies on the fact that $\varphi$ and the orientation rules inducea bijection between edges in $\ov T$ and edges of $G^+$ going out of edge-vertices, such that the ends of corresponding edges coincide; hence from $\Gd$ to $G^+\backslash\{x_1,\ldots,x_5\}$ the indegree of vertices decreases by 1 for non-root vertices and is unchanged for $v_*$, while the degree is unchanged for inner vertices and decreases by 1 for outer vertices).}
Moreover, the operations required to go from the regular orientation $\cA$ to the 5c-orientation $\cO$ can clearly be performed in linear time. This concludes the proof of Theorem~\ref{thm:main}.

% Proof of Prop. 3.1 if we agree to move it here.
\subsection{Proof of Proposition \ref{prop:acyclic}}\label{sec:pf_prop_acyclic}

Suppose for contradiction that $\cO_i$ has a simple directed cycle. We put a partial order $\prec$ on the set of simple directed cycles of $\cO_i$, by declaring $C\prec C'$ if the region enclosed by $C$ is contained in the region enclosed by $C'$. Let $C$ be a directed cycle of $\cO_i$ which is minimal for this partial order.

Suppose first that there is a vertex $v_0$ in the region enclosed by $C$. For $j \in [1:5]$, let $P_j$ be the directed path starting at $v_0$ and made of the arcs in $W_j$. This directed path $P_j$ either ends at the outer vertex $v_j$ or contains a directed cycle, but in either case it does not stay within the region enclosed by $C$ (by minimality of $C$). Hence $P_{j}$ contains a directed path $P_j'$ from $v_0$ to a vertex of $C$. 
Hence $\cO_i$ contains a directed path (going through $v_0$) from a vertex of $C$ to a vertex of $C$: for instance, take the opposite of the path $P_{i-2}'$ concatenated with the path $P_i'$. This contradicts the minimality of $C$, hence \llncs{we conclude that}there is no vertex in the region enclosed by $C$. 

Thus, still by minimality of $C$, there cannot be any edge enclosed by $C$ (since every inner edge is oriented at least 1 way in $\cO_i$). Hence $C$ must be the contour of an inner face $f$ of $G$. If $C$ is a clockwise cycle, it is easy to see that the corners of $f$ are labeled either $i-1$ or $i-2$ in the 5c-labeling $\ov \Theta(\cW)$ (because the corner at a vertex $v$ of $f$ must be incident to the arc of $W_{i+1}$ with initial vertex $v$). But this is impossible since the 3 corners of~$f$ must have distinct labels in the 5c-labeling $\ov \Theta(\cW)$. Similarly, if $C$ is a counterclockwise cycle, then the corners of~$f$ are all labeled either $i+1$ or $i+2$ in $\ov \Theta(\cW)$, which is impossible. 

This completes the proof that $\cO_i$ is acyclic. The second statement in Proposition \ref{prop:acyclic} follows immediately, which completes the proof.

% Time complexity analysis
\subsection{Time complexity}\label{appendix-time-complexity}
In this subsection we explain how the 5c-barycentric drawing algorithm can be performed in a number of operations which is linear in the number of vertices. The proof follows the same line as for Schnyder's classical algorithm.

First recall from Theorem~\ref{thm:main} that for a 5c-triangulation $G$ with $n$ vertices, a 5c-wood $\cW=(W_1,\ldots,W_5)$ can be computed in $O(n)$ operations. As we now explain, the regions' sizes $|R_i(v)|$ for all inner vertices $v$ and $i\in\set5$ can also be computed in $O(n)$ operations, hence the total time-complexity of the drawing algorithm is $O(n)$. 
%\addllncs{The linear time computation is based on  following observation: for an inner vertex $v$, the vertices strictly inside the region $R_i(v)$ are all the descendants in the tree $W_i$ of the inner vertices on the paths $P_{i-2}(v)$ and $P_{i+2}(v)$ (excluding the vertices on the paths $P_{i-2}(v)$ and $P_{i+2}(v)$). Computing the number of such descendants is possible in linear time, similarly as in the case of Schnyder's algorithm \cite{Schnyder:wood2}.}

The computation below is based on the following observation: for an inner vertex $v$, the vertices strictly inside the region $R_i(v)$ are all the descendants in the tree $W_i$ of the inner vertices on the paths $P_{i-2}(v)$ and $P_{i+2}(v)$ (excluding the vertices on the paths $P_{i-2}(v)$ and $P_{i+2}(v)$).
%$$for a 5c-triangulation $G$ with $n$ vertices, endowed with a 5c-wood $\cW=(W_1,\ldots,W_5)$, with $V$ the set of inner vertices of $G$, the time complexity of the drawing algorithm, i.e., the complexity of computing all sizes $|R_i(v)|$ over $v\in V$ and $i\in\set5$, is $O(n)$. 
Let $V$ be the set of inner vertices of $G$. For $v\in V$ and $i\in\set5$, let $N_i(v)$ be the number of descendants of $v$ in $W_i$ (with $v$ not considered a descendant of itself). 
The list $\{N_i(v),\ v\in V\}$ can be computed in $O(n)$ operations proceeding from leaves to root, that is, starting from the set $L$ of leaves of $W_i$, then the set of leaves of $W_i\backslash L$, etc. 
For $j\in\{i-2,i+2\}$, we let $N_i^j(v)$ be $N_i(w)+1$ if $v$ has a child $w$ in $W_i$ such that the arc $(v,w)$ has color $j$ (such a child is necessarily unique by Condition (W2) of 5c-woods), and be $0$ otherwise. Clearly the list $\{N_i^j(v),\ v\in V\}$ can be computed in $O(n)$ operations. 
Once all these lists are computed, for $i\in\set5$ and $v\in V$, and with the notation $P_k'(v)=P_k(v)\backslash\{v,v_k\}$, 
 we let $\Nil(v)=\sum_{u\in P_{i-2}'(v)}(N_{i}(u)-N_{i}^{i-2}(u))$, and $\Nir(v)=\sum_{u\in P_{i+2}'(v)}(N_{i}(u)-N_{i}^{i+2}(u))$.    
 The list $\{\Nil(v),\ v\in V\}$ (resp. $\{\Nir(v),\ v\in V\}$) can again be computed in $O(n)$ operations, here proceeding from root to leaves in $W_{i-2}$ (resp. in $W_{i+2}$). 
 Then, for $v\in V$, the above observation implies that the number $n_i(v)$ of vertices strictly inside $R_i(v)$ is 
 \[
 n_i(v)=\big(N_i(v)-N_i^{i-2}(v)-N_i^{i+2}(v)\big)+\Nil(v)+\Nir(v).
 \]
One can also compute the list $\{\mathrm{length}(P_i(v)),\ v\in V\}$ in $O(n)$ operations (proceeding from root to leaves in $W_i$). By the Euler relation the number of faces $|R_i(v)|$ is $2n_i(v)+\mathrm{length}(P_{i-2}(v))+\mathrm{length}(P_{i+2}(v))-1$. Hence the list $\{|R_i(v)|,\ v\in V\}$ can be computed in $O(n)$ operations, as claimed.

}

\end{document}